\documentclass[11pt]{article}

\usepackage{amsmath,amssymb}
\usepackage[margin=1in]{geometry}
\usepackage{graphicx}

\usepackage{amsthm}
\usepackage{natbib}

\usepackage{amsfonts}
\usepackage{mathtools}
\newtheorem{theorem}{Theorem}[section]
\newtheorem{lemma}[theorem]{Lemma}
\newtheorem{proposition}[theorem]{Proposition}
\newtheorem{corollary}[theorem]{Corollary}
\newtheorem{definition}{Definition}[section]
\newtheorem{assumption}{Assumption}[section]
\theoremstyle{remark}
\newtheorem{remark}{Remark}
\usepackage{hyperref}

\usepackage[utf8]{inputenc}

\def\eps{\varepsilon}
\def\C{\mathbb{C}}
\def\D{\mathbb{D}}
\def\E{\mathbf{E}}
\def\R{\mathbb{R}}
\def\Z{\mathbb{Z}}
\def\Pr{\mathbf{P}}
\def\cF{\mathcal{F}}
\def\cG{\mathcal{G}}
\def\cL{\mathcal{L}}
\def\cN{\mathcal{N}}
\def\cR{\mathcal{R}}
\def\cU{\mathcal{U}}
\def\cV{\mathcal{V}}
\def\cE{\mathcal{E}}
\def\pop{p^{\mathrm{op}}}
\def\pcl{p^{\mathrm{cl}}}
\def\rad{\mathrm{rad}}

\newcommand{\eqn}[2]{\begin{equation}\label{#1}#2\end{equation}}
\newcommand{\eqnst}[1]{\begin{equation*}#1\end{equation*}}
\newcommand{\eqnspl}[2]{\begin{equation}\begin{split}\label{#1}%
   #2\end{split}\end{equation}}
\newcommand{\eqnsplst}[1]{\begin{equation*}\begin{split}%
   #1\end{split}\end{equation*}}

\title{Scaling and partial universality of the height zero probability in the 2D Abelian sandpile}

\author{Miles W. Elvidge\thanks{School of Mathematical Sciences, 
University of Lancaster, Lancaster, LA1 4YR, United Kingdom,
Email: \texttt{m.elvidge1@lancaster.ac.uk}} \ 
and Antal A. J\'arai\thanks{Department of Mathematical Sciences,
University of Bath, Claverton Down, Bath, BA2 7AY, United Kingdom,
Email: \texttt{a.jarai@bath.ac.uk}}}

\begin{document}

\maketitle

\begin{abstract}
We study the extent of universality in the scaling of the height $0$ probability of
the stationary Abelian sandpile on lattice approximations of a region 
$U \subset \C$ with lattice spacing $\eps$ and with open boundary conditions. 
We show that under certain symmetry assumptions on the lattice, in the scaling 
limit $\eps \to 0$ this probability at a point $z \in U$ equals 
$p_\cG(0) + \eps^2 c_\cG f_U(z) + O(\eps^3)$, where $p_\cG(0)$ 
is the height $0$ probability on the full lattice, $c_\cG > 0$ is a constant 
depending only on the lattice, and $f_U$ is a conformally covariant positive 
real-valued function depending only on $U$. This generalises a result of 
Brankov, Ivashkevich and Priezzhev (1993), who considered the square 
lattice and the upper half plane. It also generalizes Example 4.10 in the recent 
study of the fermionic DGFF of Adame-Carillo and Ruszel (2025).
We show via counterexamples that our symmetry assumption cannot be omitted,
although conformal covariance may be recovered via an unusual scaling. 
In particular, the natural analogue of our result fails on general isoradial graphs. 
Our work is motivated by Jeng, Piroux and Ruelle (2006), who computed the 
correction terms for all height variables in the special case of the upper 
half plane in $\Z^2$. They conjectured that their results hold more generally 
and the present paper is a step in an attempt to make this conjecture rigorous.
\end{abstract}

\tableofcontents

\section{Introduction}

Since its introduction in \cite{BTW87,Dhar1990}, the Abelian sandpile model attracted 
a great deal of attention; see the surveys 
\cite{Dhar-survey,Redig06,Holroyd08,sandpilemodels,Ruelle21}. 
The 2D model is particularly interesting due to its connection with conformal invariance 
\cite{PirouxRuelle04a,PirouxRuelle04b,JPR06,Ruelle21}. While some aspects of conformal 
invariance have been established rigorously, particularly pertaining to height-$0$ 
variables and dissipation 
\cite{Durre2009,KasselWu15,PonceletRuelle17,PonceletRuelle18,A-CR25},
others, such as multipoint correlations between heights different from $0$, remain without
rigorous proof, despite concrete conjectures \cite{Ruelle21}.

Motivated by computations of \cite{JPR06,PonceletRuelle17,PonceletRuelle18}, in this paper 
we re-examine the height-$0$ probability from the point of view of universality. Our main 
finding, explained further in the next section, is that there is \emph{partial universality:}
while the expected conformally invariant behaviour holds on many lattices with sufficient
symmetry, it fails on others that would be natural candidates. While writing up our results,
we learned of the recent work \cite{A-CR25} on the fermionic DGFF on $\Z^2$, that, 
among many other results, gives a proof of formula \eqref{e:asymp-formula} of our main 
theorem, in the special case of the square lattice and simply connected regions. 
Our approach in this paper is more probabilistic, and we stress that it works on 
lattices without a discrete complex structure. 

In the next section we recall the definition of the model, some background, and state 
our results.

\subsection{Background and main result}
\label{ssec:main-result}

We now briefly define the Abelian sandpile model; the reader will find a lot more background in the 
sources listed above. Consider a finite connected multigraph $G = (V \cup \{s\}, E)$ with 
distinguished vertex $s$ called the \textbf{sink}. Each non-sink vertex $x$ carries a number of 
indistinguishable particles (chips, grains of sand) $\eta(x)$. When $\eta(x) \ge \deg(x)$, 
the vertex $x$ is allowed to \textbf{topple}, and send one particle along each edge incident with it.
Any particles reaching the sink are removed from the system. A certain finite sequence of topplings 
will result in a \textbf{stable} configuration, that is with $\eta(x) < \deg(x)$ for all $x \in V$.
The \textbf{sandpile Markov chain} is the discrete time Markov chain with state space the set of stable 
configurations, and one-step transition mechanism defined as follows. At each step, a vertex $x \in V$ is
selected uniformly at random, a particle is added at $x$, and the configuration is stabilised via topplings,
if needed. (The order of topplings does not matter \cite{Dhar1990}.)

The case when $V$ is a finite subset of a Euclidean lattice is of particular interest. In this case, 
we typically collapse the complement of $V$ to a single vertex, $s$, and remove resulting loop-edges 
at $s$. We will refer to this as the \textbf{open boundary condition}, since particles can leave anywhere 
along the boundary, and write $\nu_V$ for the stationary 
distribution of the sandpile Markov chain to indicate the dependence on $V$, when the underlying lattice 
is fixed. It follows from results of \cite{AJ04,Jarai2014minimal} that when the underlying lattice is at least 
two-dimensional, then there exists a measure $\nu$ such that for any sequence of finite subsets
$V_1 \subset V_2 \subset \dots$ whose union is the entire lattice, $\nu_{V_n}$ converges weakly to
$\nu$ as $n \to \infty$. In particular we have 
\eqnsplst
{ p(i) 
  := \nu( \eta(o) = i )
  = \lim_{n \to \infty} \nu_{V_n} ( \eta(o) = i ), \quad i = 0, \dots, \deg(o)-1. }
Such limiting measures could be defined more generally, on certain infinite subgraphs 
of the lattice, for example, the discrete upper half plane in $\Z^2$, using ideas from
\cite{AJ04,Jarai2014minimal}. This is also possible with 
\textbf{partially closed boundary conditions}, where particles cannot cross a segment
of the boundary. While the existence of such limits is relevant for the motivating 
discussion in the next paragraph, we do not include a proof here, because for the 
present paper a weaker statement suffices; see Lemma \ref{lem:limit-unbounded}.

In order to motivate our results, we briefly describe a discovery of Jeng, Piroux and Ruelle \cite{JPR06}
from 2006. They performed calculations on the upper half plane on the 
square lattice and discovered an intriguing relationship between the height variables. One of their 
results can be described as follows. Let $\pop_m(i)$, $i=0,1,2,3$ denote the probability 
that the height-variable at a vertex at microscopic distance $m$ from the boundary 
equals $i$ with open boundary conditions.\footnote{Note that \cite{JPR06} labels the 
heights $1,2,3,4$. Here we follow the convention to label these $0,1,2,3$.} 
Let $\pcl_m(i)$ denote the same probability with closed boundary conditions.
Let $p(i)$ denote the probability that the height-variable at the origin 
on the full plane equals $i$. Then these probabilities satisfy an asymptotic relation of the form
\eqnspl{e:JPR-rel}
{ \pcl_m(i) 
	&= p(i) - \frac{1}{m^2} \left( a_i + b_i \log m \right) + o(m^{-2}), 
    \quad i = 0, 1, 2, 3, \quad \text{as $m \to \infty$}; \\
	\pop_m(i) 
	&= p(i) + \frac{1}{m^2} \left( a_i + \frac{b_i}{2} + b_i \log m \right) + o(m^{-2}), 
      \quad i = 0, 1, 2, 3, \text{as $m \to \infty$}, }
with explicit constants $a_i, b_i$, where $b_0 = 0$.
The values of the constants found in \cite{JPR06} imply that up to an error $o(m^{-2})$, the expressions 
$p^{\mathrm{op/cl}}_m(2) - p(2)$ and $p^{\mathrm{op/cl}}_m(3) - p(3)$ are \emph{linear combinations} of
$p^{\mathrm{op/cl}}_m(0) - p(0)$ and $p^{\mathrm{op/cl}}_m(1) - p(1)$.
Jeng, Piroux and Ruelle conjecture in \cite{JPR06} that this phenomenon
should hold in the scaling limit 'in any geometry and with any sort of boundary condition'.

In this paper we take a step towards a rigorous analysis of the above conjecture, 
by studying the scaling of the height $0$ probability in general regions and its 
universality. This generalises a result of \cite{BIP-93}.
Extending the results to the other height variables remains the 
subject of future work. 

In order to state our first main result, let us first clarify our assumptions on the underlying graphs we consider.
\begin{assumption}
\label{a:graph}\ \\
(i) Let $\cG = (\cV,\cE)$ be a multigraph (without loop-edges) with $\cV$ embedded 
into $\C \cong \R^2$ in a $\Z^2$-periodic way, that is, there exist two linearly independent vectors
$\mathbf{h}_1, \mathbf{h}_2 \in \mathbb{R}^2$ such the translations 
by $n_1 \mathbf{h}_1 + n_2 \mathbf{h}_2$, $(n_1,n_2) \in \mathbb{Z}^2$ leave $\cG$
invariant. We assume that $o \in \cV$, where $o$ denotes 
the origin in $\C$. We assume the embedding is locally finite, that is, every 
compact subset of $\C$ contains finitely many vertices. We think of the edges  
$e \in \cE$ represented by closed line segments $\overline{xy}$, where 
$x,y$ are the endpoints of $e$. \\
(ii) Assume that the embedding of $\cG$ is invariant under a rotation around 
every $v \in \cV$ by an angle $2 \pi/q$ for some integer $q = q_v \ge 3$.
\end{assumption}
Note that $\Z^2$-periodicity allows quasi-transitive graphs, such as examples obtained 
from the square or hexagonal lattice by adding extra edges in a way that keeps 
the rotational symmetry; see Figure \ref{fig:example-lattices}. 
It can be deduced that $q_v$ is restricted to the values
$\{ 3, 4, 6 \}$ by the assumptions (i)--(ii): a rotation by $2 \pi / q$, 
has trace $2 \cos(2 \pi / q)$, and this must be rational due to the $\Z^2$-periodicity.
This is only the case for $q \in \{ 1, 2, 3, 4, 6 \}$ \cite{L33}. Let $\mathcal{O}$ 
denote the set of orbits under the translation symmetry group. For a vertex $x \in \cV$ 
we denote by $[x] \in \mathcal{O}$ its orbit.
\begin{figure}
    \centering
    \includegraphics[width=0.4\linewidth]{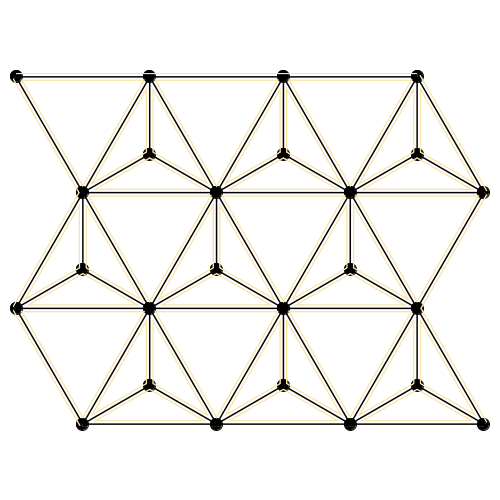} \qquad
    \includegraphics[width=0.4\linewidth]{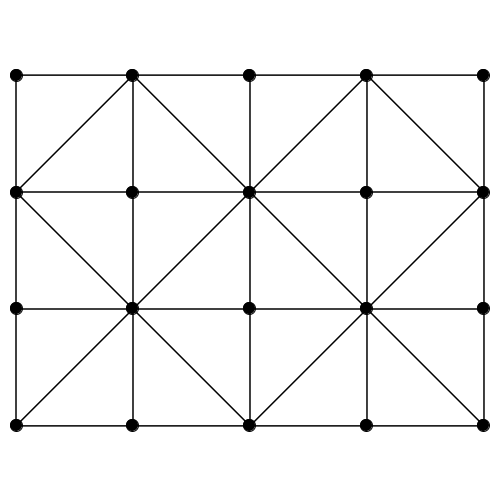}
    \caption{Examples of lattices satisfying Assumption \ref{a:graph} that are not
    transitive.}
    \label{fig:example-lattices}
\end{figure}

Next we state our assumptions on the regions we consider. 
\begin{assumption}
\label{a:U} 
Let $U \subsetneq \mathbb{C}$ be a region (connected open set) such that for each point
$a \in \partial_\infty U$ (considered in the Riemann sphere) there is a continuous simple
curve starting at $a$ and lying outside $U$ (i.e.~there is an injective 
$\gamma: [0,1] \to U^c$ with $\gamma(0) = a$.
It will be convenient for the proof to distinguish the regions we consider as follows. \ \\
(i) $U$ is bounded; \\
(ii) $U$ is unbounded with at least one unbounded boundary component.
\end{assumption}

\begin{remark}
\label{rem:bdry-U}
Our assumption on $U$ is slightly stronger than what is required to ensure that the 
Dirichlet problem in $U$ (with continuous boundary data) can be solved 
uniquely \cite[Corollary 10.4.17]{Conway-book-I}. We assume the stronger condition 
to facilitate a convergence theorem for discrete harmonic functions; 
see Theorem \ref{thm2_3}.
\end{remark}

We now define the lattice approximation of $U$ we consider.
Given a region $U \subsetneq \C$ satisfying Assumption \ref{a:U} and given $\eps > 0$, 
let $U_\eps := U \cap \eps \cV$. We turn 
$U_\eps$ into a graph by letting $E_\eps$ denote those edges of $\eps \cE$ that do not
have a point in common with $\partial U$. In order to define the boundary edges, consider
any $e \in \eps \vec{\cE}$ such that $e^- \in U_\eps$ and 
$\overline{e^- e^+} \cap \partial U \not= \emptyset$. We define 
\[ \partial_E U_\eps
   := \{ e : \text{$e^- \in U_\eps$ and 
      $\overline{e^- e^+} \cap \partial U \not= \emptyset$} \}. \]
Note: if both $e^-, e^+ \in U_\eps$ but $\overline{e^- e^+} \cap \partial U \not= \emptyset$,
then $e$ and $\overleftarrow{e}$ represent different elements of $\partial_E U_\eps$. (This 
will be the case for regions with a slit boundary piece.) We define the multigraph 
$\cU_\eps = (U_\eps \cup \{ s \}, E_\eps \cup \partial_E U_\eps)$, where each element
$e \in \partial_E U_\eps$ is now regarded as having endpoints $e^-$ and $s$.
The multigraph $\cU_\eps$ is non-empty for sufficiently small $\eps$, 
and connected (due to the sink $s$). 
For each $z \in U$, fix a lattice point $z_\eps \in U_\eps \cap [o]$ nearest to $z$.

Let $p_{U,\eps}(0;z)$ denote the probability that in the stationary sandpile model 
on the graph $\cU_\eps$ with sink $s$, the height at $z_\eps$ equals $0$. (In the case 
of unbounded $U$, refer to Lemma \ref{lem:limit-unbounded} for well-definedness.) 
It follows from existence of the potential kernel on $\cG$ (see Lemma \ref{lem:MD-trick})  
that the limit $p_\cG(0) = \lim_{\eps \to 0} p_{U,\eps}(0;z)$ exists. 
Let $\D$ denote the unit disc in $\C$.

\begin{theorem}
\label{thm:scaling-form}
Let $\cG$ satisfy Assumption \ref{a:graph}.  
Suppose that $U \subsetneq \C$ satisfies Assumption \ref{a:U}. \\
(i) There is a constant $c_\cG$, only depending on $\cG$, and a real-valued function $f_U(z)$, only depending on $U$, such that
\eqn{e:asymp-formula}
{ p_{U,\eps}(0;z)
  = p_\cG(0) + c_\cG f_U(z) \eps^2 + \mathcal{O}(\eps^3), \quad \text{as $\eps \to 0$.} }
(ii) The function $f_U(z)$ is conformally covariant with scale dimension $2$. \\
(iii) If $U$ is simply-connected, then
\eqn{e:conf-rad}
{ f_U(z)
  = 2 |\phi'_U(z)|^{2}, }
where $\phi_U : U \to \D$ is (any) conformal map such that $\phi_U(z) = o$. \\
(iv) We have $f_U > 0$ and $c_G > 0$. 
\end{theorem}

\begin{remark} \ \\
(a) The quantity $|\phi'_U(z)|^{-1} =: \rad(U,z)$ is known as the \textbf{conformal radius} 
of $U$ viewed from $z$. \textbf{Conformal covariance with scale dimension $\alpha$} is the
statement that if $\psi : U \to V$ is a conformal map then 
$f_U(z) = |\psi'(z)|^{\alpha} f_V(\psi(z))$. As shown in \cite{A-CR25}, $f_U$ can also be 
expressed in terms of the Green function $g_U(z,w)$ of $U$. Let 
$k_U(z,w) = g_U(z,w) + \log|z-w|$. Then 
\[ f_U(z)
   = \bar{\partial}^{(1)} \partial^{(2)} k_U(z,z)
   = \partial^{(1)}_x \partial^{(2)}_x k_U(z,z) 
     + \partial^{(1)}_y \partial^{(2)}_y k_U(z,z), \]
where $z = x + i y$, and $\partial^{(1)}_x$, etc.~refer to partial derivative with respect to
$x$ in the first variable, etc. This formula remains valid for non-simply connected $U$.\\
(b) When $z$ is not restricted by the condition $[z]=[o]$, and $\cG$ is only quasi-transitive
(but satisfies Assumption \ref{a:graph}), the same proof yields the modified statement where $p_\cG(0)$ and $c_\cG$ also depend on the orbit of $z_\eps$ under the automorphism
group of $\cG$, but are otherwise independent of $z$. \\
(c) The error term $\mathcal{O}(\eps^3)$ depends on $z$, but is locally uniform. \\
(d) In Section \ref{sec:ex-and-counter-ex} we compute $c_\cG$ in several examples. 
\end{remark}

Our second main contribution is two examples that show that the universality seen in 
Theorem \ref{thm:scaling-form} is only \emph{partial}: even when the scaling limit of random walk 
on the lattice is isotropic, the conformal radius formula is not guaranteed to hold
without some sort of additional symmetry. As the first example, consider the 
modification of the square lattice where each horizontal edge is doubled (replaced 
by two parallel edges). This lattice is only $180$-degree symmetric. After an 
appropriate stretch in the vertical direction, the scaling limit of simple random 
walk on the lattice becomes isotropic. In particular, the potential kernel on the 
stretched lattice is asymptotically isotropic at large distances. Despite this, the 
limiting function $f_U$ on this lattice is different from the cases covered by 
Theorem \ref{thm:scaling-form}; see Section \ref{ssec:180}. Interestingly, by 
\emph{changing} the vertical scaling, that is making the random walk \emph{non-isotropic},
one can recover conformal invariance. We find this surprising, since under the
Abelian sandpile the dynamics the individual particles follow random walk trajectories.

Our second counterexample concerns isoradial graphs. Works of Kenyon \cite{K02} and 
Chelkak and Smirnov \cite{CS11} showed that on critically weighted isoradial graphs the 
potential kernel is asymptotically isotropic, and the scaling limit of the weighted random walk 
is a time-change of planar Brownian motion. We show in Section \ref{ssec:isorad}, 
using the residue formulas of Kenyon \cite{K02}, that with the natural generalisation 
of the height-$0$ probability to weighted graphs, again the simple form 
of the function $f_U$ fails on general isoradial graphs.

The observation that the Abelian sandpile shows non-universal (that is, lattice dependent)
behaviour is not new. The fractal patterns seen in the scaling limit of the single source
sandpile are lattice dependent \cite{LPS17}. See also \cite[Section 10]{LS24} that contrasts 
properties of the Activated Random Walk with those of the Abelian sandpile. An interesting 
feature of our results, however, is that their lattice dependence is only partial.

We note here some questions that remain open. \\
1. Can the results be extended to the other height variables, in a way 
that allows a check of the conjecture of \cite{JPR06}? \\
2. Is there a precise characterisation, or at least a more general sufficient condition 
on the local symmetry required, for the conclusion of Theorem \ref{thm:scaling-form} to hold?
Is there an appropriate generalisation to lattices where it fails? \\
3. Can the results be extended to other boundary conditions, for example when part of the boundary is open and part of it is closed? \\

\emph{Outline of the paper.} 
The starting point for our proof is Majumdar and Dhar's method \cite{Majumdar1992,Dhar1991} 
for representing the full plane probability that we recall in Section \ref{sec:prelim}. 
We analyse the resulting expression via the asymptotic 
expansion of the potential kernel due to Fukai and Uchiyama \cite{FU96} in 
Section \ref{sec:pot-kern}.
The main result for bounded regions will then be achieved via an averaging over 
rotations of the graphs in Section \ref{sec:proof-main}. In Section \ref{sec:unbounded} 
we extend the result to unbounded regions. In Section \ref{sec:positive} we prove positivity
of $c_\cG$ and $f_U$. Finally, in Section \ref{sec:ex-and-counter-ex} 
we discuss examples and counterexamples. Some extensions we need of well-known results 
are given in the appendices.

\section{Preliminaries}
\label{sec:prelim}

\subsection{The finite volume Green function}
\label{ssec:finite-Green}
We denote by 
\eqnst{
  \Delta_{vw}
  = \begin{cases}
    \deg(v) & \text{if $w = v$;} \\
    - a_{vw} & \text{if $w \not= v$,}
  \end{cases} }
the graph Laplacian of $\cG$, where $\deg(v)$ denotes the degree of $v$ and $a_{vw}$ 
is the number of edges between $v$ and $w$ in $\cG$. When $a_{vw} > 0$, we write $v \sim w$.
There is a slight subtlety in defining the restriction of $\Delta$ to finite subgraphs
stemming from how we defined boundaries of sets in Section \ref{ssec:main-result}. 
We will have to work with pairs $(\Lambda, \partial \Lambda)$ of the form
$(\eps^{-1} U_\eps, \eps^{-1} \partial_E U_\eps)$. We will refer to such pairs as a
\emph{graph with boundary}. When $\Lambda \subset \cV$ is 
finite, we denote by $\left( \Delta_\Lambda \right)_{vw}$, $v,w \in \Lambda$, the 
\textbf{Dirichlet Laplacian}, that is, the matrix indexed by vertices in $\Lambda$ 
having entries
\[ \left( \Delta_\Lambda \right)_{vw}
   = \begin{cases}
     \deg(v) & \text{if $w=v$;} \\[2ex]
     - a_{vw} & \parbox{5.5cm}{$v \sim w$ and for the edge(s) $e$ with endpoints $v,w$ 
        we have $e, \overleftarrow{e} \not\in \partial \Lambda$;} \\[2ex]
     0    & \parbox{5.5cm}{$v \sim w$ and for the edge(s) $e$ with endpoints $v,w$ 
        we have $e, \overleftarrow{e} \in \partial \Lambda$.}
   \end{cases} \]
(Note: the dependence of the matrix $\Delta_\Lambda$ on $\partial \Lambda$ is suppressed.)

We let $G_\Lambda$ denote the finite volume Green function, defined as the inverse:
\eqnst{
  G_\Lambda(v,w) 
  = \left( \Delta_\Lambda^{-1} \right)_{vw}; }
which is a symmetric matrix. We have
\eqnst{
  G_\Lambda(v,w) 
  = \frac{1}{\deg(w)} \mathbf{E}^v \left[ \sum_{n=0}^{\tau_\Lambda-1} \mathbf{1}_{S_n = w} \right], }
where $(S_n)_{n \ge 0}$ is a simple random walk on $\cG$, and $\tau_\Lambda$ is the first 
exit time. It will sometimes be necessary, in the case of parallel edges, to keep track 
of which edge the simple random walk crosses: let $E_1, E_2, \dots \in \mathcal{E}$ be 
the sequence of edges crossed by the walk, where $E_n$ has endpoints $S_{n-1}$ and $S_n$. 
We then have
\[ \tau_\Lambda 
   = \inf \{ n \ge 0 : \text{$S_n \not\in \Lambda$ or 
     $E_n \in \partial \Lambda$} \}. \]
We interpret $G_\Lambda(v,w)$ as $0$, when at least one of $v, w$ is in 
$\cV \setminus \Lambda$.

The potential kernel of $\cG$ is defined by
\eqnst{
  A(w,v) 
  = \lim_{N \to \infty} \sum_{n = 0}^N \left[ \frac{p_n(w,w)}{\deg(w)} - \frac{p_n(w,v)}{\deg(v)} \right], }
where $p_n(w,v) = \mathbf{P}^w [ S_n = v ]$.
The limit exists, for example, by \cite{KU08} and satisfies
\eqnspl{e:A-apriori}{
  \Delta A (w,\cdot)
  &= - \delta_w(\cdot), }
where 
\[ \delta_w(v) 
   = \begin{cases}
     1 & \text{if $v = w$;} \\
     0 & \text{if $v \not= w$.}
   \end{cases} \]
We also have that
\eqn{e:A-nonneg}{
  A(w,v)
  \ge 0, \quad v, w \in \cV. }

The following well-known lemma can be found, for example, in \cite[Proposition 4.6.2]{LL2010}. 
Note that although the notion of lattice used in \cite{LL2010} is 
different from what we assume, the same proof works.

\begin{lemma}
\label{lem:finite-Green}
For a finite graph with boundary $(\Lambda, \partial \Lambda)$, we have 
\eqnst{
  G_\Lambda(v,w)
  = -A(v,w) + \sum_{e = (e^-,e^+) \in \partial \Lambda} H_\Lambda(v,e) A(e^+,w), 
    \quad v, w \in \Lambda, }
where $H_\Lambda(v,e) = \mathbf{P}^v [ E_{\tau_\Lambda} = e ]$.
\end{lemma}

\subsection{Majumdar and Dhar's method}
\label{ssec:MD-trick}

In this section we briefly recall the modification trick of Majumdar and Dhar \cite{Dhar1991,Majumdar1992} 
in a form that will be convenient for our purposes. Note that from 
Lemma \ref{lem:MD-trick} below we will only 
need to use part (ii); see \cite{Jarai2014minimal} for this version of the result.

Let $(\Lambda, \partial \Lambda)$ be a finite graph with boundary, 
and assume $o \in \Lambda$. Let $\cG_\Lambda$ denote 
(suppressing the dependence on $\partial \Lambda$) 
the wired graph induced by $(\Lambda, \partial \Lambda)$, that is, 
the graph with vertex set $\Lambda \cup \{ s \}$ obtained by identifying the tails
of boundary edges to a single vertex, denoted $s$ (the 'sink'), and removing loops 
created at $s$. We write $\mathsf{UST}_{\cG_\Lambda}$ for the probability measure 
on spanning subgraphs of $\cG_\Lambda$ that gives equal weight to all spanning trees 
and no weight to other subgraphs. Let $g$ be any edge incident with $o$, with other 
endpoint $z$. Let $\cN^{(g)}$ denote the set
of edges incident with $o$, other than $g$. Let the graph $\cG^{(g)}_\Lambda$ 
be obtained from $\cG_\Lambda$ by deleting all edges in $\cN^{(g)}$. 

The transfer-current matrix $Y_\Lambda$ can be defined by giving each edge of $\cG_\Lambda$ an 
arbitrary orientation, and setting 
\eqnsplst{
  Y_\Lambda(e,f)
  &:= \partial^1_e \partial^2_f G_\Lambda(x,y) 
  := \big[ G_\Lambda(u,v) - G_\Lambda(u,y) - G_\Lambda(x,v) + G_\Lambda(x,y) \big], \\ 
  &\qquad\quad \text{where $e_- = x$, $e_+ = u$, $f_- = y$, $f_+ = v$.} }
(See \cite{LyonsTrees} for background on the transfer-current matrix.)

\begin{lemma}[See \cite{Dhar1991,Majumdar1992,Jarai2014minimal}.]
\label{lem:MD-trick} \ \\
(i) We have 
\eqn{e:MD-trick}
{ p_\Lambda(0; o)
  = \det ( I + G_\Lambda B ), }
where $B$ is the matrix defined by $\Delta_{\cG^{(g)}_\Lambda} = \Delta_\Lambda + B$. \\
(ii) The above can also be written in the form
\eqn{e:MD-transf-curr}
{ p_\Lambda(0; o)
  = \det ( I - Y^{(g)}_\Lambda ), }
where $Y^{(g)}_\Lambda$ is the submatrix of $Y_\Lambda$ obtained by deleting the rows and columns 
corresponding to all edges not incident with $o$, as well as the row and column corresponding
to $g$.
\end{lemma}

\begin{proof}[Sketch of proof.]
Using the burning bijection of Majumdar and Dhar \cite{Majumdar1992}, we have that 
\eqnspl{e:UST-prob}{
  p_\Lambda(0; o)
  &= \frac{|\cR(\cG^{(g)}_\Lambda)|}{|\cR(\cG_\Lambda)|}
  = \mathsf{UST}_{\cG_\Lambda} [ \text{the edges in $\cN^{(g)}$ are not present} ] \\
  &= \frac{1}{\deg(o)} \mathsf{UST}_{\cG_\Lambda} [ \text{$o$ is a leaf} ]. }
The second expression equals $\det(\Delta_{\cG^{(g)}_\Lambda})/\det(\Delta_\Lambda)$,
yielding part (i). The third expression can be expressed using the Transfer-Current 
Theorem of Burton and Pemantle \cite{Burton1993}, \cite[Exercise 4.41]{LyonsTrees} as
\eqnst{
  \det( I - Y^{(g)}_\Lambda ). }
\end{proof}

\subsection{Potential kernel asymptotics}
\label{sec:pot-kern}

The following proposition plays a fundamental role in our computations. As explained below, 
it readily follows from \cite[Eqn.~(15)]{KU08} of Kazami and Uchiyama. We remark that, 
exploiting the symmetry assumptions we made, the proposition can also be derived directly 
from \cite[Theorem 2]{FU96} of Fukai and Uchiyama, without relying on the general result 
of \cite[Theorem 2]{U07} on Markov additive processes. We provide such a derivation in
Appendix \ref{sec:appendix}.

Before stating the proposition, we need to introduce some notation. 
Let $\mathbf{h_1}, \mathbf{h_2} \in \mathbb{R}^2$
be linearly independent periodicity vectors of $\cV$, and let $M$ be the $2 \times 2$ matrix 
that has $\mathbf{h_1}, \mathbf{h_2}$ as its columns. Let $T$ denote a fixed 
(necessarily finite) fundamental 
set of the sublattice generated by $\mathbf{h_1}, \mathbf{h_2}$, i.e.~a set of 
representatives of the orbits of the automorphisms of $\cV$ generated by the translations 
by the vectors $\mathbf{h_1}$ and $\mathbf{h_2}$. The random walk on $\cV$ induces a
Markov chain on $T$. Let 
\[ \mu(u) 
   = \frac{\deg(u)}{\sum_{v \in T} \deg(v)}, \quad u \in \cV, \]
denote its invariant distribution.
We will write $[u]$ for the equivalence class of $u$ according to the sublattice spanned by
$\mathbf{h_1}, \mathbf{h_2}$. 

Define the $2 \times 2$ symmetric matrix $Q$ and the quadratic form $Q(\theta)$ via
\[ \theta \cdot Q \theta
   := Q(\theta)
   := \sum_{u \in T} \mu(u)
      \mathbf{E}^u \left[ (\theta \cdot (S_1 - S_0))^2 \right]
   = \frac{|\theta|^2}{2} \sum_{u \in T} \mu(u) \E^u \left[ |S_1-S_0|^2 \right], \]
where in the last equality we used rotational symmetry. We also define
\[ \sigma^2
   := (\det Q)^{1/2}
   = \frac{1}{2} \sum_{u \in T} \frac{\deg(u)}{\sum_{s \in T} \deg(s)} 
     \frac{1}{\deg(u)} \sum_{y \sim u} |y - u|^2
   = \frac{\sum_{u \in T} \sum_{y \sim u} |y - u|^2}{2 Z_T}, \]
where, for the last equality, we defined
\[ Z_T
   := \sum_{s \in T} \deg(s). \]
   
\begin{proposition}[{See \cite[Eqn.~(15)]{KU08}}]
\label{prop:A-asymp}
Under Assumption \ref{a:graph}, there exist $c_0 = c_0(\cG)$ and $C_0 = C_0(\cG)$ 
and a smooth function $U_2$ on the unit circle (also dependent on $\cG$), such that 
\eqnst{
  A(w,v)
  = \frac{c_0}{\pi} \log |v - w| + C_0 + \frac{U_2((v - w)/|v - w|)}{|v - w|^2} 
    + O \left( |v - w|^{-3} \right), }
as $| v - w | \to \infty$. We have $c_0 = |\det(M)|/ Z_T \sigma^2$.
\end{proposition}

\begin{proof}
The forms of the main and constant terms follow immediately from \cite[Theorem 1.2]{KU08}.
(Observe that while in that theorem the constant term $C_0$ can depend on $[w]$, in our case reversibility
gives that it does not.)
In the same theorem,
it is easy to see that the coefficient of the term of order $|v - w|^{-1}$ vanishes (in the notation of
\cite{KU08}, the vectors $c$ and $c^*$ vanish, due to rotational symmetry). The form of the term of order
$|v - w|^{-2}$ follows from \cite[Eqn.~(15)]{KU08} and the discussion following it.
\end{proof}

\subsection{Harmonic Functions}
\label{ssec:harmonic}

We recall some properties of harmonic functions and their discrete analogues. 

\begin{definition}
\label{def2_1}
    Let $U\subset\mathbb{R}^2$ be an open connected region with boundary \(\partial U\). A continuous function \(f : U\cup\partial U \rightarrow \mathbb{R}\) is said to be \it{harmonic} in \(U\) if it is twice continuously differentiable and satisfies 
    \[
        \Delta^{(c)} f = 0 \; \text{ in } U.
    \]
    For a given continuous function \(g : \partial U \rightarrow \mathbb{R}\), \(f\) is said to be the solution to Laplace's equation with Dirichlet boundary condition \(g\) if, in addition, it satisfies
    \[
        f = g \; \text{ on } \partial U.
    \]
\end{definition}

\begin{theorem}[{See \cite[Corollary 10.4.17]{Conway-book-I}}]
\label{thm2_1}
Let $U \subset\mathbb{R}^2$ be a connected region satisfying Assumption \ref{a:U},
and let \(g : \partial U \rightarrow \mathbb{R}\) be bounded and continuous.
There exists a unique function \(u : U\cup\partial U \rightarrow \mathbb{R}\) that 
is harmonic in \(U\) and satisfies the Dirichlet boundary condition \(u = g\) 
on \(\partial U\). 
\end{theorem}

We recall a discrete analogue to harmonic functions.
\begin{definition}
\label{def2_2}
    Let $U\subset\mathbb{R}^2$ be a connected region with boundary \(\partial U\), satisfying
    Assumption \ref{a:U}. Fix \(\epsilon > 0\).
    \begin{itemize}
        \item Let $f : \epsilon\cV \rightarrow \mathbb{R}$ and let $e = (z,w)$ be an oriented edge of $(\eps\cV, \eps \cE)$. 
        Define the discrete derivative $\partial_e$ to be the operator given by 
            \[
                (\partial_e f)(x) = f(w) - f(z) 
                = f(z+e) - f(z), \quad x \in \eps\cV.
            \] 
        \item We define the discrete Laplace operator $\Delta$
            \[
                (\Delta f)(z) = \sum_{e \in \eps\cE : e_- = z} (f(z)-f(e_+)) 
                = -\sum_{e} \partial_e f(z), \quad z\in\eps\cV.
            \]
        \item A function \(f : U_\epsilon \cup \partial U_\epsilon \rightarrow \mathbb{R}\) is said to be discrete harmonic in \(U_\epsilon\) if it satisfies \(\Delta f = 0\) in \(U_\epsilon\). (For the purposes of computing $(\Delta f)(z)$ when $z \in U_\eps$ 
        and $e \in \partial U_\eps$ with $e_- = z$, the value of $f$ at $e_+$ is taken to be $f(e)$.)
        \item For a given function \(g : \partial U_\epsilon \rightarrow \mathbb{R}\), $f$ is said to be the solution to the discrete Laplace equation with Dirichlet boundary condition \(g\) if, in addition, \(f = g\) on \(\partial U_\epsilon\).
        \end{itemize}
\end{definition}
The following is a well-known discrete version of Theorem \ref{thm2_1}.
\begin{theorem}[{\cite[Theorem 6.2.1]{LL2010}}]
\label{thm2_2}
    Assume $U$ is bounded. Fix \(\epsilon > 0\) and 
    let \(g : \partial U_\epsilon \rightarrow \mathbb{R}\). 
    There exists a unique function $u : U_\epsilon \cup \partial U_\epsilon\rightarrow \mathbb{R}$ that is discrete harmonic in \(U_\epsilon\) and that satisfies the Dirichlet boundary conditions \(u = g\) on \(\partial U_\epsilon\). Moreover, the solution is given by 
    \[
        u(z) = \sum_{f \in\partial U_\epsilon} H_{\partial U_\epsilon}(z,f) g(f), 
             \quad z \in U_\epsilon, \qquad\qquad u(f) = g(f), \quad f \in \partial U_\eps.
    \]
    where \(H_{\partial U_\epsilon}(z, \cdot) : \partial U_\epsilon \rightarrow \mathbb{R}\) denotes the exit distribution, or Poisson kernel, for the simple random walk on \(\epsilon\cV\) started at \(z\in U_\epsilon\), stopped on first exiting \(U_\epsilon\). 
    That is, denoting the simple random walk by \((S_n)_{n\geq 0}\), and the stopping 
    time \(\tau_{U_\epsilon} = \inf\{n \geq 0 : E_n \in \partial U_\epsilon\}\), we have 
    \[
        H_{\partial U_\epsilon}(z,f) 
        = \Pr^z[E_{\tau_{\partial U_\epsilon}} = f ], \quad
          z\in U_\epsilon, \, f \in \partial U_\epsilon.
    \]
\end{theorem}
We will need that discrete harmonic functions, together with their partial derivatives, approach their continuous analogue as the mesh \(\epsilon\downarrow 0\). We state next a version of the classical result \cite{Courtant1928} that we will use in computations. Specifically, we will 
use part (ii) of the theorem.

\begin{theorem}
\label{thm2_3}
    Let $U\subset\mathbb{R}^2$ be a bounded connected region satisfying Assumption \ref{a:U}(i). Let \(g : W\rightarrow \mathbb{R}\) be a $C^1$ function on an open set \(W\) containing \(\partial U_\epsilon\) for all \(\epsilon > 0\) sufficiently small. Let \(u : U \cup \partial U \rightarrow \mathbb{R}\) denote the solution to the Laplace equation on \(U\) with Dirichlet boundary condition \(g\). For \(\epsilon > 0\) such that \(\partial U_\epsilon \subset W\), let \(u_\epsilon : U_\epsilon \cup \partial U_\epsilon \rightarrow \mathbb{R}\) denote the solution to the discrete Laplace equation on \(U_\epsilon\) with Dirichlet boundary condition \(u_\epsilon = g\) on \(\partial U_\epsilon\). \\
    (i) The family of functions \((u_\epsilon)\) satisfy
    \[
        u_\epsilon = u + \mathcal{O}(\eps) \; \text{ as } \epsilon \downarrow 0,
    \]
    where the error term is locally uniform. Moreover, the error term is also uniform 
    in $U$, as long as $|g|$ remains under a fixed bound, and the distance from 
    $\partial U$ is bounded away from $0$.\\
    (ii) Furthermore, for \(z \in U\) and \(e \in \eps \cE\) an edge incident 
    with $z_\eps$, we have 
    \[
        \partial_e u_\epsilon(z_\epsilon) 
        = u_\epsilon(z_\epsilon + e)-u_\epsilon(z_\epsilon)
        = \eps \tilde{e} \cdot \nabla u (x) + O(\eps^2), \; \text{as } \epsilon \downarrow 0,
    \]
    where $\tilde{e} = e/\eps$ is the corresponding vector in $\cE$. 
    The error term satisfies the same uniformity statement as in part (i).
\end{theorem}

While the proof is essentially available elsewhere, we provide a summary of it in 
Appendix \ref{sec:proof-harmonic-conv}, since the uniformity claims will
be important for us to extend to unbounded regions.

\section{Bounded regions}
\label{sec:proof-main}

In this section, we prove Theorem \ref{thm:scaling-form}(i) for bounded $U$.

\subsection{Asymptotics of the transfer-current matrix entries}
\label{ssec:transf-current}

For convenience of notation in the next proposition, let us introduce the following abbreviations. For 
$z, w \in U_\eps$, $g \in \partial U_\eps$, and edges $e, f$ of $U_\eps$ incident with $x$, let
\eqnsplst{
  \hat{e}
  &= \eps^{-1} e \qquad\qquad \hat{f} = \eps^{-1} f \\
  G_{\eps}(z,w)
  &:= G_{\eps^{-1} U_\eps} (\eps^{-1} z, \eps^{-1} w), \\
  A_\eps(z,w)
  &:= A(\eps^{-1} z, \eps^{-1} w), \\
  H_{\eps}(z,g)
  &:= H_{\eps^{-1} U_\eps} (\eps^{-1} z, \eps^{-1} g) \\
  Y_{\eps}(e,f)
  &= Y_{\eps^{-1} U_\eps} ( \hat{e}, \hat{f} ). }
We will also need a quadratic form $J_{e,f}$ defined in the following way. 
Let $g_U(z,w)$ denote the Green function of $U$, and let 
$k_U(z,w) = g_U(z,w) + \log|z-w|$, which is symmetric and 
harmonic in both variables \cite[Chapter 10]{Conway-book-I}. Let $Q_{U}(z; \cdot, \cdot)$ 
denote the quadratic form obtained by taking the gradient of $k_U$ in both variables, and 
evaluating it at $(z,z)$. Put
\eqn{e:Jef-def}
{  J_{e,f}
   = J_{e,f,z,U}
   = Q_U(z;e,f)
   = - (\hat{e} \cdot \nabla^{(1)})(\hat{f} \cdot \nabla^{(2)}) k_U (z,z). }
Note that $J_{e,f} = J_{f,e}$, as follows from the symmetry $g_U(z,w) = g_U(w,z)$.
For later use, also note that for each $w \in U$, the 
function $z \mapsto (\hat{f} \cdot \nabla^{(2)}) k_U(z,w)$ equals the harmonic function 
in $U$ with boundary values $z \mapsto - \Re (\hat{f}/(z - w))$.

\begin{proposition}
\label{prop:Y-asymp} 
Suppose that $\cG$ satisfies Assumption \ref{a:graph} and that $U$ satisfies 
Assumption \ref{a:U}.
For any $z \in U_\eps$ and edges $e,f$ incident with $z$, we have 
\eqnspl{e:Yeps-asymp}{
 Y_{\eps}(e,f)
 = Y(\hat{e}, \hat{f}) - \eps^2 \frac{c_0}{\pi} J_{e,f} + O(\eps^3), }
where $c_0$ is the constant from Proposition \ref{prop:A-asymp}, and
where the error term is locally uniform in $z$, when $z$ is viewed as a point in $U$.\\
\end{proposition}

\begin{proof}
(i)    We recall from Lemma \ref{lem:finite-Green} that for $z, w \in U_\eps$ we have
    \[
        G_{\epsilon}(z,w) = -A_\eps(z,w) + \sum_{e \in \partial U_\epsilon} H_{\epsilon}(z,e) A_\epsilon(e^+,w).
    \]
    Hence it follows that 
    \eqnsplst{
        Y_\eps(e,f)
        &= \partial_e^{(1)} \partial_f^{(2)} G_{\epsilon}(z,z) \\
        &= -\partial_e^{(1)} \partial_f^{(2)} A_\epsilon(z,z) 
          + \sum_{g \in \partial U_\epsilon} \partial_e^{(1)}
            H_{\epsilon}(z,g) \partial_f^{(2)} A_\epsilon(g^+,z) \\
        &= Y( \hat{e}, \hat{f} ) 
          + \sum_{g \in \partial U_\epsilon} \partial_e^{(1)}
            H_{\epsilon}(z,g) \partial_f^{(2)} A_\epsilon(g^+,z). }
    Thus we need to estimate \(\partial_f^{(2)} A_\epsilon(g^+,z)\). 
    We use Proposition \ref{prop:A-asymp} to get
    \[
        A_\eps(g^+,z) =  \frac{c_0}{\pi} \log \left|\frac{z - g^+}{\eps}\right| + C_0 
        + \eps^2 \frac{U_2((z - g^+)/|z - g^+|)}{|z - g^+|^{2}} 
        + \mathcal{O}(\eps^3), 
    \]
    Since $|z + f - z| = \mathcal{O}(\eps)$ and $U_2$ is smooth, we have that
    \[ \frac{U_2((z + f - g^+)/|x + f - g^+|)}{|z + f - g^+|^{2}} - 
       \frac{U_2((z - g^+)/|z - g^+|)}{|z - g^+|^{2}} 
       = \mathcal{O}(\eps). \]
    Hence we have that 
    \eqnsplst{
        \partial_f^{(2)} A_\epsilon(g^+, z) 
        &= A(\epsilon^{-1} g^+, \epsilon^{-1} (z + f)) 
        - A(\epsilon^{-1} g^+, \epsilon^{-1} z) \\
        &= \frac{c_0}{\pi} \left( \log \left| \frac{z + f - g^+}{\eps} \right| 
          - \log \left| \frac{z - g^+}{\eps} \right| \right) + \mathcal{O}(\epsilon^3). }
    Writing $w = z - g^+$ and regarding $f$ and $w$
    as complex numbers, as \(\epsilon \rightarrow 0\) we have that
    \eqnsplst{
        \log |w + f| - \log|w| 
        &= \log | 1 + f w^{-1}| 
        = \log | 1 + \eps \hat{f} (z-g^+)^{-1} | \\
        &= \frac{1}{2} \log \left( 1 + 2 \eps \Re( \hat{f} (z-g^+)^{-1} ) 
           + \eps^{2} |\hat{f} (z-g^+)^{-1}|^2 \right) \\
        &= - \eps \Re( \hat{f} (g^+-z)^{-1} ) + \eps^{2} a(z,g^+,\hat{f}) 
           + \mathcal{O}(\eps^{3}), }
    where 
    \[ a(z,g^+,\hat{f})
       = |\hat{f} (z-g^+)^{-1}|^2 - 2 \Re( \hat{f} (z-g^+)^{-1} ). \]
    It follows that 
    \eqnsplst{
        &\sum_{g \in \partial U_\epsilon} \partial_e^{(1)} H_{\eps}(z,g) \partial_f^{(2)} A_\epsilon(g^+,z) \\
        &\quad = - \eps \frac{c_0}{\pi} \sum_{g \in \partial U_\epsilon} \Re \left( \frac{\hat{f}}{g^+-z} \right) 
        \partial_e^{(1)} H_{\epsilon}(z,g)  
           + \eps^2 \frac{c_0}{\pi} \sum_{g \in \partial U_\epsilon} a(z,g^+,\hat{f})
             \partial_e^{(1)} H_{U_\epsilon}(z,g) + \mathcal{O}(\eps^{3}). }
    Recalling the remark made after \eqref{e:Jef-def}, and 
    applying Theorem \ref{thm2_3}, the first term on the right hand side equals
    \[ - \eps^2 \frac{c_0}{\pi} J_{e,f} + \mathcal{O}(\eps^{3}). \]
    By the same theorem, the second term is $\mathcal{O}(\eps^{3})$, 
    and the proposition follows.
\end{proof}

\subsection{Computation of the determinant}
\label{ssec:comp-det}

Throughout this section, we fix $z \in U$ and denote by $O$ the rotation about $z_\eps$ 
by an angle $2\pi/q$ that leaves $\cG$ invariant.
We need some simple properties of the quantity $J_{e, f}(z)$ arising from the fact that 
it is bilinear in $e, f$, 
when $e, f$ are viewed as vectors in $\R^2 \equiv \C$. Let us write 
$\mathbf{1}, \mathbf{i} \in \C$, and define
\eqnsplst{
 J_{11} &= J_{\mathbf{1}, \mathbf{1}} \qquad\qquad J_{22} = J_{\mathbf{i}, \mathbf{i}} \\
 J_{12} &= J_{\mathbf{1}, \mathbf{i}} \qquad\qquad J_{21} = J_{\mathbf{i}, \mathbf{1}}. }
 
\begin{lemma}
\label{lem:Javrg}
We have
\eqnsplst{
 \frac{1}{q} \sum_{r=0}^{q-1} J_{O^r e, O^r f} 
 = \frac{1}{2} \langle e, f \rangle (J_{11} + J_{22})
   + \frac{1}{2} (e_1 f_2 - e_2 f_1) (J_{12} - J_{21})
 = \frac{1}{2} \langle e, f \rangle (J_{11} + J_{22}), }
where $\langle e, f \rangle = e_1 f_1 + e_2 f_2$, with $e = e_1 \mathbf{1} + e_2 \mathbf{i}$ (and similarly 
for $f$).
\end{lemma}

\begin{proof}
The second equality is immediate from $J_{12} = J_{21}$. To see the first,
from the definition of $J_{e,f}$ in \eqref{e:Jef-def} we have
\eqnspl{e:Jef}
{ J_{e, f} 
  = e_1 f_1 J_{11} + e_2 f_2 J_{22} + e_1 f_2 J_{12} + e_2 f_1 J_{21}. }
Due to \eqref{e:Jef} we get
\eqnst
{ J_{O^r e, O^r f}
  = (O^r e)_1 (O^r f)_1 J_{11} + (O^r e)_2 (O^r f)_2 J_{22} 
    + (O^r e)_1 (O^r f)_2 J_{12} + (O^r e)_2 (O^r f)_1 J_{21}. }
Summing the coefficient of $J_{11}$ over $r$ we get 
\eqnsplst{
 &\sum_{r=0}^{q-1} (O^r e)_1 (O^r f)_1 \\
 &\quad= \sum_{r=0}^{q-1} \left[ \cos \left( \frac{r 2 \pi}{q} \right) e_1 
   - \sin \left( \frac{r 2 \pi}{q} \right) e_2 \right]
   \left[ \cos \left( \frac{r 2 \pi}{q} \right) f_1 
   - \sin \left( \frac{r 2 \pi}{q} \right) f_2 \right] \\
 &\quad= \sum_{r=0}^{q-1} \left[ \cos^2 \left( \frac{r 2 \pi}{q}\right) e_1 f_1
   + \sin^2 \left( \frac{r 2 \pi}{q}\right) e_2 f_2 
   - \cos \left( \frac{r 2 \pi}{q} \right) \sin \left( \frac{r 2 \pi}{q} \right) (e_2 f_1 + e_1 f_2) \right]. }
Using that when $q \ge 3$, we have
\eqnsplst{
  \sum_{r=0}^{q-1} \cos^2 \left( \frac{r 2 \pi}{q}\right)
  = \frac{q}{2}, \qquad\quad
  \sum_{r=0}^{q-1} \sin^2 \left( \frac{r 2 \pi}{q}\right)
  = \frac{q}{2}, \qquad\quad
  \sum_{r=0}^{q-1} \frac{1}{2} \sin \left( \frac{r 4 \pi}{q}\right)
  = 0, }
we get that the coefficient of $J_{11}$ equals
\eqnsplst{
 \frac{q}{2} (e_1 f_1 + e_2 f_2) 
 = \frac{q}{2} \langle e,f \rangle. }
An analogous computation for the coefficient of $J_{22}$ gives the same expression.

For $J_{12}$, we compute 
\eqnsplst{
  &\sum_{r=0}^{q-1} (O^r e)_1 (O^r f)_2 \\
  &\quad = \sum_{r=0}^{q-1} \left[ \cos \left( \frac{r 2 \pi}{q} \right) e_1 
   - \sin \left( \frac{r 2 \pi}{q} \right) e_2 \right]
   \left[ \sin \left( \frac{r 2 \pi}{q} \right) f_1 
   + \cos \left( \frac{r 2 \pi}{q} \right) f_2 \right] \\
  &\quad= \sum_{r=0}^{q-1} \left[ \cos^2 \left( \frac{r 2 \pi}{q}\right) e_1 f_2
   - \sin^2 \left( \frac{r 2 \pi}{q}\right) e_2 f_1 
   + \cos \left( \frac{r 2 \pi}{q} \right) \sin \left( \frac{r 2 \pi}{q} \right) (e_1 f_1 - e_2 f_2) \right] \\
  &\quad= \frac{q}{2} (e_1 f_2 - e_2 f_1). }
We get the negative of the same expression for the coefficient for $J_{21}$, and this completes the proof.
\end{proof}

\begin{proof}[Proof of Theorem \ref{thm:scaling-form}(i) under Assumption \ref{a:U}(i).]
Recall from Lemma \ref{lem:MD-trick}(ii) that
\[ p_{U_\eps}(0;z)
   = \det(I - Y^{(g)}_{\eps}), \]
where $g$ is an edge incident with $z_\eps$, $\cN^{(g)}$ denotes the set of edges incident with $z_\eps$
in $U_\eps$ other than $g$, and 
\[ Y^{(g)}_\eps(e,f)
   = Y_\eps(e,f), \quad e, f \in \cN^{(g)}. \]
Write $\cN$ for the set of edges incident with $z_\eps$ in $U_\eps$. It will be convenient
to order the set $\{ g \} \cup \mathcal{N}$ in such a way that its cyclic order remains
invariant under the action of the rotation $O$. (For example, by increasing argument of the
endpoint.) Let $M^{(g)}_{e,f}$ denote the minor of $(I - Y)_{e, f \in \cN}$ with 
the rows $g,e$ and columns $g,f$ deleted. Then due to Proposition \ref{prop:Y-asymp} we can write
\eqnsplst{
 p_{U_\eps}(0;z)
 = p(0) + \eps^2 \frac{c_0}{\pi} \sum_{e,f \not= g} 
   J_{e,f}(z) (-1)^{n(e)+n(f)} M^{(g)}_{e,f} + \mathcal{O}(\eps^3), }
where $n(e), n(f)$ denotes the indices of $e, f$, respectively, in the fixed ordering.
Since the choice of $g$ is free, we can also write
\begin{equation*}  
\begin{split}
 p_{U_\eps}(0;z)
 &= \frac{1}{q} \sum_{r=0}^{q-1} \det(I - Y^{(O^r g)}_{\eps}) \\
 &= p(0) + \eps^2 \frac{c_0}{\pi} \frac{1}{q} \sum_{r=0}^{q-1} \sum_{e,f \not= g} 
   J_{O^r e, O^r f} (-1)^{n(O^r e) + n(O^r f)} M^{(O^r g)}_{O^r e, O^r f}
   + \mathcal{O}(\eps^3). 
\end{split}
\end{equation*}
It is easy to see that $(-1)^{n(O^r e) + n(O^r f)} = (-1)^{n(e) + n(f)}$, and 
by rotation invariance of the lattice we also have 
$M^{(O^r g)}_{O^r e, O^r f} = M^{(g)}_{e,f}$ (as the matrices differ by a cyclic permutation 
of both the rows and columns). This gives:
\eqnspl{pcorr}
{ &\frac{1}{q} \sum_{r=0}^{q-1} \sum_{e,f \not= g} 
     J_{O^r e, O^r f} (-1)^{n(O^r e) + n(O^r f)} M^{(O^r g)}_{O^r e, O^r f} \\
  &\qquad = \sum_{e,f \not= g} (-1)^{n(e) + n(f)} M^g_{e,f} 
     \left[ \frac{1}{q} \sum_{r=0}^{q-1} J_{O^r e, O^r f} \right]. }
Adding the $J_{11}$ and $J_{22}$ contributions from Lemma \ref{lem:Javrg}, and substituting 
into \eqref{pcorr}, we have the expression
\eqnspl{e:c_0c_1}{
  \frac{c_0}{\pi} (J_{11} + J_{22}) \frac{1}{2} 
     \sum_{e,f \not= g} (-1)^{n(e)+n(f)} M^{(g)}_{e,f} \langle e,f \rangle
  =: \frac{c_0}{\pi} c_1 (J_{11} + J_{22}). } 
Putting $f_U = (J_{11} + J_{22})$ and $c_\cG := c_0 c_1 / \pi$ completes the proof.
\end{proof}

\begin{remark}
When the boundary of $U$ is smooth, we have the following explicit formula for $f_U$:
\eqn{e:fU-formula}{
   f_U(z)
   = J_{11}+J_{22}
   = \int_{\partial U} \left[\Re\left(\frac{1}{w-z}\right) \frac{\partial h_U(z,w)}{\partial x}-\Im\left(\frac{1}{w-z}\right) \frac{\partial h_U(z,w)}{\partial y}\right] d\sigma(w), }
where $z = (x,y)$, $\sigma$ is the length measure on $\partial U$ and $h_U$ is the density of 
harmonic measure with respect to $\sigma$.
\end{remark}

\section{Unbounded regions}
\label{sec:unbounded}

In this section we prove Theorem \ref{thm:scaling-form}(i) for unbounded $U$, under Assumption
\ref{a:U}(ii).

We start with showing that $p_{U,\eps}(0;z)$ is well defined as a limit obtained from
finite subgraphs of $U_\eps$. For $R \ge 1$, consider the bounded region $U_R$ defined
as the connected component of $U \cap B_R(z)$ that contains $z$, and let $U_{R,\eps}$
denote its lattice approximation.

\begin{lemma}
\label{lem:limit-unbounded}
For any $\eps > 0$, we have
\eqn{e:limit-unbounded}
{ p_{U,\eps}(0;z)
  = \lim_{R \to \infty} p_{U_R,\eps}(0;z). }
\end{lemma}

\begin{proof}
The statement follows immediately from the formula \eqref{e:MD-trick} applied with 
$\Lambda = U_{R,\eps}$ and the convergence 
$\lim_{R \to \infty} G_{U_{R,\eps}}(u,v) = G_{U_\eps}(u,v) < \infty$, where $u,v$ 
range over $z_\eps$ and its neighbours. Here $G_{U_\eps}$ is the Green kernel of 
simple random walk on $\eps \cV$ killed on its first crossing of $\partial U_\eps$. 
This is finite due to part of Assumption \ref{a:U} concerning $\partial U$.
\end{proof}

\begin{proof}[Proof of Theorem \ref{thm:scaling-form}(i) under Assumption \ref{a:U}(ii).]
Applying the already proved bounded case to $U_R$, we have
\eqn{e:bounded-asymp-formula}
{ p_{U_R,\eps}(0;z)
  = p(0) + c_\cG f_{U_R}(z) \eps^2 + \mathcal{O}_R(\eps^3), \quad \text{as $\eps \to 0$.} }
We now argue that $\lim_{R \to \infty} f_{U_R}(z) = f_U(z)$, where $f_U(z)$ is
defined by the integral in \eqref{e:fU-formula}, and that the $\mathcal{O}_R(\eps^3)$
error term is in fact uniform in $R$. 

For the convergence of $f_{U_R}$, we use a representation in terms of planar Brownian motion. 
Let $r(z) = \inf \{ |y - z| : y \in \partial U \}$. Define
the function
\[ F_U(z) 
   = \E^z \left[ \Re\left(\frac{1}{B_{\xi_{U}}-z}\right) \right], \]
where $B$ is a standard planar Brownian motion, and where
$\xi_U = \inf \{ t \ge 0 : B_t \not\in U \}$. Together with the analogous formula 
for $F_{U_R}$, we have 
\[ F_U(z) - F_{U_R}(z)
   = \E^z \left[ \Re\left(\frac{1}{B_{\xi_{U}}-z}\right) 
     - \Re\left(\frac{1}{B_{\xi_{U_R}}-z}\right) \right]. \]
As a consequence of the mean-value property of harmonic 
functions \cite[Theorem 2.2.7]{Evans-book} we have that
\[ \left| \frac{\partial F_U}{\partial x} - \frac{\partial F_{U_R}}{\partial x} \right| 
   \le \frac{C}{r(z)} \max_{w : |w - z| = r(z)/2} |F_U(w) - F_{U_R}(w)|, \]
where $x = (x,y)$. Now we have 
\[ |F_U(w) - F_{U_R}(w)|
   \le \left( \max_{v \in \partial U \Delta \partial U_R} \Re\left(\frac{1}{v-w}\right) \right) 
       \mathbf{P}^w \left[ \xi_U \not= \xi_{U_R} \right] 
   \le \frac{C}{r(z)} \mathbf{P}^w \left[ \xi_U \not= \xi_{U_R} \right]. \]
We show that the last probability goes to $0$, as $R \to \infty$, uniformly in
$w$ such that $|w - z| = r(z)/2$. From Assumption \ref{a:U} we have that
we can find a curve $\gamma \subset U^c$ with $\mathrm{diam}(\gamma) > 0$, 
such that for all sufficiently large $R$, we have $\gamma \subset U_R^c$ as well. 
The probability that $B$ does not visit $\gamma$ before exiting $B_R(z)$ goes to $0$ 
(uniformly over $w$), and from this the statement follows. Thus we have 
$\frac{\partial F_{U_R}}{\partial x} \to \frac{\partial F_{U}}{\partial x}$.

A very similar argument can be applied with $\Re\left(\frac{1}{v-z}\right)$ replaced by
$\Im\left(\frac{1}{v-z}\right)$ and $\frac{\partial}{\partial x}$ replaced by 
$\frac{\partial}{\partial y}$, and it follows that $f_{U_R}(z) \to f_U(z)$.

Uniformity of the error term in $R$ follows from the uniform boundedness of the boundary data $1/(v-z)$.
\end{proof}

\section{Conformal covariance}
\label{sec:conformal}

In this section, we prove that the formula obtained in Section \ref{ssec:comp-det} is 
conformally covariant with scale dimension 2.

\subsection{Transformation formula}
\label{ssec:conf}

\begin{theorem}
\label{thm3_6}
    Let \(U, V \subsetneq \mathbb{C}\) be regions satisfying Assumption \ref{a:U}, and assume that
    $\varphi : U \to V$ is a conformal map. Let $z \in U$ and let $w := \varphi(z) \in V$. 
It holds that 
\eqnst{
 f_U(z)
 = |\varphi'(z)|^2 f_V(\varphi(z)). }
\end{theorem}

\begin{proof}
Conformal invariance of the Green function \cite[Theorem X.5.3]{Conway-book-I} gives that 
\eqnsplst{ 
  k_U(z,w)
  &= g_U(z,w) + \log |z - w|
  = g_V(\phi(z),\phi(w)) + \log |z - w| \\
  &= k_V(\phi(z),\phi(w)) - \log \left| \frac{\phi(z) - \phi(w)}{z-w} \right|. }
Denote $F(z,w) = (\phi(z) - \phi(w))/(z-w)$, which is holomorphic in the $z$-variable
in $U$, and non-vanishing. This implies that $\bar{\partial}^{(1)} [ \log F ] = 0$, 
and hence
\[ \bar{\partial}^{(1)} \partial^{(2)} \log |F(z,w)|
   = \bar{\partial}^{(1)} \left[ \partial^{(2)} \Re \log F(z,w) \right] 
   = \bar{\partial}^{(1)} \left[ i \partial^{(2)} \Im \log F(z,w) \right], \]
where we use the Cauchy-Riemann equations in the last step. From the last equality 
it follows that 
\[ \bar{\partial}^{(1)} \partial^{(2)} \log |F(z,w)|
   = \frac{1}{2} \bar{\partial}^{(1)} \partial^{(2)} \left[ \log F(z,w) \right]
   = \frac{1}{2} \partial^{(2)} \left[ \bar{\partial}^{(1)} \log F(z,w) \right]
   = 0. \]
It remains to compute $\bar{\partial}^{(1)} \partial^{(2)} k_V( \phi(z), \phi(w) )$. 
Denote by $\partial^{(1)}_{Re}$ the partial derivative of $k_V$ in the real part of the 
first variable, and similarly write $\partial^{(1)}_{Im}$, $\partial^{(2)}_{Re}$, $\partial^{(2)}_{Im}$ for the other partial derivatives. Then we have
\eqnsplst{
  &\bar{\partial}^{(1)} \partial^{(2)} k_V (\phi(z), \phi(w)) \\
  &\qquad = \left( \partial^{(1)}_{Re} \partial^{(2)}_{Re} k_V \right) 
     \left( \bar{\partial} \Re \phi(z) \right) \left( \partial \Re \phi(w) \right) 
     + \left( \partial^{(1)}_{Im} \partial^{(2)}_{Im} k_V \right) 
     \left( \bar{\partial} \Im \phi(z) \right) \left( \partial \Im \phi(w) \right). }
Here setting $w = z$ gives
\[ \left( \bar{\partial} \Re \phi(z) \right) \left( \partial \Re \phi(z) \right)
   = \left( \partial_x \Re \phi(z) \right)^2 + \left( \partial_y \Re \phi(z) \right)^2
   = \left| \phi'(z) \right|^2, \]
and likewise
\[ \left( \bar{\partial} \Im \phi(z) \right) \left( \partial \Im \phi(z) \right)
   = \left( \partial_x \Im \phi(z) \right)^2 + \left( \partial_y \Im \phi(z) \right)^2
   = \left| \phi'(z) \right|^2. \]
Noting that 
\[ \left( \partial^{(1)}_{Re} \partial^{(2)}_{Re} k_V \right) 
   + \left( \partial^{(1)}_{Im} \partial^{(2)}_{Im} k_V \right) 
   = - f_V(\phi(z)) \]
completes the proof.
\end{proof}

\subsection{Conformal radius formula}

We recall the definition of the \textbf{conformal radius} of a simply connected region \(U\subset\mathbb{C}\).
\begin{definition}\label{def3_1}
    Let \(U\subset\mathbb{C}\) be a simply connected region. Let \(z\in U\) and let \(\varphi : U \rightarrow \mathbb{D}\) be the unique conformal mapping to the unit disk \(\mathbb{D}\coloneqq \{w \in \mathbb{C} : |w| < 1\}\) satisfying \(\varphi(z) = 0\), \(\varphi'(z) > 0\). The conformal radius of \(U\) viewed from \(z\) is defined to be 
    \[
        \mathrm{rad}_U(z) = \frac{1}{\varphi'(z)}.
    \]
\end{definition}

A corollary of Theorem \ref{thm3_6} is an extension to the lattices we consider of the 
formula of \cite[Example 4.10]{A-CR25}.

\begin{corollary}\label{cor3_7}
    Let \(U\subset\mathbb{C}\) be a simply connected region satisfying Assumption \ref{a:U}. Let \(z\in U\).  
Then 
    \[  f_U(z)
        = \frac{2}{(\mathrm{rad}_U(z))^2}.   \]
\end{corollary}

\begin{proof}
Due to Theorem \ref{thm3_6}, it is enough to show that $f_{\mathbb{D}}(o) = 2$.
Since $g_{\mathbb{D}}(z,w) = - \log ( |w - z| / |1 - z \bar{w}| )$, we have
$k_{\mathbb{D}}(z,w) = \log | 1 - z \bar{w} |$. A direct computation then yields
the result.
\end{proof}

\section{Positivity of $f_U$ and $c_G$}
\label{sec:positive}

Let $T_{U_\eps}$ denote the UST in $U_\eps$, and let $T$ denote the UST in $\cG$. 

\begin{lemma}
\label{lem:monotone} \ \\
(i) If $z \in U$, we have $f_U(z) \ge 0$. \\
(ii) If $z \in U \subset V$, we have $f_U(z) \ge f_V(z)$.
\end{lemma}

\begin{proof}
(i) Let $e$ be an edge incident with $z_\eps$ in $U_\eps$. By the Transfer-Current Theorem and 
Proposition \ref{prop:Y-asymp}, we have
\eqnsplst{
 \frac{1}{q} \sum_{r=0}^{q-1} \mathbf{P} [O^r e \in T_{U_\eps}]
 &= \frac{1}{q} \sum_{r=0}^{q-1} Y_\eps(O^r e, O^r e) \\
 &= Y(\hat{e}, \hat{e}) - \eps^2 \frac{1}{q} \sum_{r=0}^{q-1} J_{O^r e, O^r e} + \mathcal{O}(\eps^3). }
Due to Lemma \ref{lem:Javrg}, the right hand side equals
\eqnspl{e:fU-expr}{
 Y(\hat{e}, \hat{e}) - \eps^2 \frac{|\hat{e}|^2}{2} (J_{11} + J_{22}) + \mathcal{O}(\eps^3)
 = Y(\hat{e}, \hat{e}) - \eps^2 |\hat{e}|^2 f_U(z) + \mathcal{O}(\eps^3). }
On the other hand, by monotonicity in the graph, see \cite[Exercise 10.8]{LyonsTrees}, we have
\eqnsplst{
 \frac{1}{q} \sum_{r=0}^{q-1} \mathbf{P} [O^r e \in T_{U_\eps}]
 \le \frac{1}{q} \sum_{r=0}^{q-1} \mathbf{P} [O^r \hat{e} \in T]
 = \mathbf{P} [\hat{e} \in T]
 = Y(\hat{e}, \hat{e}). }
Letting $\eps \to 0$ we get $f_U(z) \ge 0$, proving part (i).

(ii) Due to monotonicity in the graph we have
\eqnsplst{
 \frac{1}{q} \sum_{r=0}^{q-1} \mathbf{P} [O^r e \in T_{U_\eps}]
 \le \frac{1}{q} \sum_{r=0}^{q-1} \mathbf{P} [O^r e \in T_{V_\eps}], }
and the claim follows from \ref{e:fU-expr} letting $\eps \to 0$.
\end{proof}

\begin{lemma}
\label{lem:fU>0}
Suppose that $U$ satisfies Assumption \ref{a:U}. If $z \in U$, we have $f_U(z) > 0$.
\end{lemma}

\begin{proof}
First assume that $U$ is bounded.
Take $V \supset U$ to be an open disc in Lemma \ref{lem:monotone}(ii). Then by Corollary \ref{cor3_7}, 
we have $f_V(z) > 0$, and the claim follows from Lemma \ref{lem:monotone}(ii).

In the unbounded case, we can map $U$ conformally to a bounded region satisfying Assumption
\ref{a:U}(i). Indeed, there is a simple curve $\gamma$ in the Riemann sphere $\mathbb{C}_\infty$
starting at $\infty$ such that $U \subset V := \mathbb{C} \setminus \gamma$. Then a conformal
mapping of $V$ to a bounded simply connected region exists, and this maps $U$ conformally 
to a bounded region. Now the claim follows from Theorem \ref{thm3_6}.
\end{proof}

It remains to prove that $c_G > 0$.

\begin{lemma}
\label{lem:asymp-diff}
There exists $c = c(\cG,U,z) > 0$ such that for all sufficiently small $\eps > 0$ we have
\eqn{e:asymp-diff}
{ p_{U,\eps}(0;z) 
  \ge p(0) + c \eps^2. }
\end{lemma}

\begin{proof}
Let $\cN = \{ e_1, \dots, e_N \}$ be an enumeration of the set of edges incident with $z_\eps$.
Averaging as in the proof of Lemma \ref{lem:monotone}(i), we have that
\[ \frac{1}{q} \sum_{r=0}^{q-1} \mathbf{P} [O^r e_1 \not\in T_{U_\eps}]
   = 1 - Y(\hat{e_1}, \hat{e_1}) + \eps^2 |\hat{e_1}|^2 f_U(z)
     + \mathcal{O}(\eps^3). \]
Hence, for sufficiently small $\eps > 0$, there exists $0 \le r < q$ such that 
\[ \mathbf{P} [O^r e_1 \not\in T_{U_\eps}]
   \ge 1 - Y(\hat{e_1}, \hat{e_1}) + \frac{1}{2} \eps^2 |\hat{e_1}|^2 f_U(z). \]
Without loss of generality, assume that $r=0$, so that
\eqn{e:e_1-ineq}{
 \mathbf{P} [e_1 \not\in T_{U_\eps}]
   \ge 1 - Y(\hat{e_1}, \hat{e_1}) + \frac{1}{2} \eps^2 |\hat{e_1}|^2 f_U(z). }

Let $U_{\eps,j}$ denote the graph obtained from $U_\eps$ by deleting $e_1, \dots, e_j$,
and let $\cG_j$ denote the graph obtained similarly from $\cG$. 
Then the left hand side of \eqref{e:asymp-diff} equals
\eqnspl{e:T-prob-prod}{
  \prod_{i=1}^{N-1} \mathbf{P} [ e_i \not\in T_{U_\eps} \,|\, e_1, \dots, e_{i-1} \not\in T_{U_\eps} ]. }
By monotonicity in the graph, see \cite[Lemma 10.3]{LyonsTrees}, for $i = 2, \dots, N-1$, we also have 
\eqnsplst{
  \mathbf{P} [ e_i \not\in T_{U_\eps} \,|\, e_1, \dots, e_{i-1} \not\in T_{U_\eps} ]
  &= \mathbf{P} [ e_i \not\in T_{U_{\eps,i-1}} ] \\
  &\ge \mathbf{P} [ e_i \not\in T_{\cG_{i-1}} ] \\
  &= \mathbf{P} [ e_i \not\in T \,|\, e_1, \dots, e_{i-1} \not\in T ]. }
This gives that \eqref{e:T-prob-prod} is at least
\eqnsplst{
  &\mathbf{P} [ e_1 \not\in T_{U_\eps} ] \prod_{i=2}^{N-1} \mathbf{P} [ e_i \not\in T \,|\, e_1, \dots, e_{i-1} \not\in T ]
  + c' \eps^2 \prod_{i=2}^{N-1} \mathbf{P} [ e_i \not\in T \,|\, e_1, \dots, e_{i-1} \not\in T ] \\
  &\qquad \ge p(0) + c' \eps^2 \frac{p(0)}{\mathbf{P} [e_1 \not\in T]}
  \ge p(0) + c \eps^2. }
\end{proof}

\begin{proof}[Proof of Theorem \ref{thm:scaling-form}(iv).]
This follows immediately from the scaling form in part (i), the estimate in Lemma \ref{lem:asymp-diff}, 
and the already established positivity of $f_U$ in Lemma \ref{lem:fU>0}.
\end{proof}

\section{Examples and counterexamples}
\label{sec:ex-and-counter-ex}

\subsection{The square, triangular, and hexagonal lattices}

We compute $c_\cG$ for the square, triangular and hexagonal lattices.

\emph{The square lattice.} 
In Proposition \ref{prop:A-asymp} we have that $M$ is the identity matrix, and 
$\sigma^2 Z_T = 2$. Thus $c_0 = 1/2$. In equation \ref{e:c_0c_1}, take 
$g = \mathbf{1} \in \mathbb{C}$, so that 
$\mathcal{N} = \{ \mathbf{i}, \mathbf{-1}, \mathbf{-i} \}$. The relevant potential kernel 
values are (see for example \cite{Sp-book}): 
\[ A(o,\mathbf{i})
   = \frac{1}{4} \qquad\qquad 
   A(\mathbf{i},\mathbf{-i})
   = 1 - \frac{2}{\pi} \qquad\qquad
   A(\mathbf{i},\mathbf{-i})
   = \frac{1}{\pi}, \]
with other required values deduced by symmetry. Thus
\[ (I - Y)_{e,f \in \mathcal{N}}
   = \begin{pmatrix}
   \frac{1}{2} & \frac{1}{\pi} - \frac{1}{2} & \frac{1}{2} - \frac{2}{\pi} \\[0.2cm]
   \frac{1}{\pi} - \frac{1}{2} & \frac{1}{2} & \frac{1}{\pi} - \frac{1}{2} \\[0.2cm]
   \frac{1}{2} - \frac{2}{\pi} & \frac{1}{\pi} - \frac{1}{2} & \frac{1}{2}
   \end{pmatrix}. \]
A trite calculation then gives
\[ c_1 = \frac{2(\pi - 2)}{\pi^2} \qquad\qquad
   c_\cG = \frac{c_0 c_1}{\pi} = \frac{\pi - 2}{\pi^3}. \]
We recover a result of \cite{BIP-93} (the case of $i = 0$ with open boundary condition in \eqref{e:JPR-rel}) by observing that for the upper half plane $\mathbb{H}$, we have
$f_{\mathbb{H}}(i) = 2 (1/2)^2 = 1/2$. This gives with $\cG = \mathbb{Z}^2$ that
\[ p_{\mathbb{H}}(0,m \mathbf{i})
   = p_{\mathbb{H},1/m}(0,\mathbf{i})
   = p_{\mathbb{Z}^2}(0) + m^{-2} \frac{2 (\pi - 2)}{\pi^3} \frac{1}{4} + O(m^{-3}), \quad \text{as $m \to \infty$.} \]

\emph{The triangular lattice.}
We fix the neighbours of the origin to be $g = \frac{1}{2} - \frac{\sqrt{3}}{2} \mathbf{i}$
and $e_1 = -\mathbf{1}$, $e_2 = \frac{1}{2} + \frac{\sqrt{3}}{2} \mathbf{i}$. We can take
$\mathbf{h_1} = \frac{3}{2} + \frac{\sqrt{3}}{2} \mathbf{i}$ and
$\mathbf{h_2} = \sqrt{3} \mathbf{i}$, which gives $|\det(M)| = \frac{3 \sqrt{3}}{2}$.
We can take $T = \{ o, e_1 \}$, and hence $Z_T \sigma^2 = 3$ and 
$c_0 = |\det(M)| / Z_T \sigma^2 = \sqrt{3}/2$.
From symmetry and harmonicity, we have
\[ A(o,e_1) = \frac{1}{3} \qquad\qquad
   A(e_1,e_2) = \frac{1}{2}. \]
Thus, with $\mathcal{N} = \{ e_1, e_2 \}$ we have
\[ (I - Y)_{e,f \in \mathcal{N}}
   = \begin{pmatrix}
   \frac{1}{3} & \frac{1}{6} \\[0.2cm]
   \frac{1}{6} & \frac{1}{3}
   \end{pmatrix}. \]
This gives
\[ c_1 
   = \frac{1}{2} \sum_{e,f \not= g} (-1)^{n(e)+n(f)} M^{(g)}_{e,f} \langle e, f \rangle 
   = \frac{1}{2} \left[ \frac{1}{3} - \frac{1}{6} \left( -\frac{1}{2} \right) - \frac{1}{6} \left( - \frac{1}{2} \right) + \frac{1}{3} \right] 
   = \frac{5}{12}, \]
and $c_\cG = c_0 c_1 / \pi = 5 \sqrt{3} / 24 \pi$.

\emph{The hexagonal lattice.} 
We fix the neighbours of the origin to be:
\[ j_1 = \mathbf{1} \qquad 
   j_2 = \frac{1}{2} + \frac{\sqrt{3}}{2} \mathbf{i} \qquad
   j_3 = -\frac{1}{2} + \frac{\sqrt{3}}{2} \mathbf{i} \qquad
   j_4 = - j_1 \qquad 
   j_5 = - j_2 \qquad
   j_6 = - j_3. \]
As periodicity vectors, we can take $\mathbf{h_1} = \mathbf{1}$ and 
$\mathbf{h_2} = j_2 = \frac{1}{2} + \frac{\sqrt{3}}{2} \mathbf{i}$, which gives
$|\det(M)| = \frac{\sqrt{3}}{2}$. We have $\Z_T \sigma^2 = \frac{1}{2}$, so
$c_0 = |\det(M)| / Z_T \sigma^2 = \sqrt{3}$.

The relevant values of the potential kernel are:
\[ A(o,j_1) = \frac{1}{6} \qquad 
   A(j_1,j_2) = \frac{\sqrt{3}}{\pi} - \frac{1}{3} \qquad
   A(j_1,j_4) = \frac{4}{3} - \frac{2 \sqrt{3}}{\pi}. \]
These can be obtained, for example, by the Fourier formula \eqref{e:A(x,y)-formula}.
Thus, with $g = j_1$ and $\mathcal{N} = \{ j_2, j_3, j_4, j_5, j_6 \}$ we have
\[ (I - Y)_{e,f \in \mathcal{N}}
   = \begin{pmatrix}
   \frac{2}{3} & - \frac{1}{6} & -\frac{2}{3} + \frac{\sqrt{3}}{\pi} 
      & 1 - \frac{2\sqrt{3}}{\pi} & -\frac{2}{3} + \frac{\sqrt{3}}{\pi} \\[0.2cm]
   -\frac{1}{6} & \frac{2}{3} & -\frac{1}{6} & -\frac{2}{3} + \frac{\sqrt{3}}{\pi}
      & 1 - \frac{2\sqrt{3}}{\pi} \\[0.2cm]
   -\frac{2}{3} + \frac{\sqrt{3}}{\pi} & -\frac{1}{6} & \frac{2}{3} 
      & -\frac{1}{6} & -\frac{2}{3} + \frac{\sqrt{3}}{\pi} \\[0.2cm]
   1 - \frac{2\sqrt{3}}{\pi} & -\frac{2}{3} + \frac{\sqrt{3}}{\pi} & -\frac{1}{6} 
      & \frac{2}{3} & -\frac{1}{6} \\[0.2cm]
   -\frac{2}{3} + \frac{\sqrt{3}}{\pi} & 1 - \frac{2\sqrt{3}}{\pi}
      & -\frac{2}{3} + \frac{\sqrt{3}}{\pi} & -\frac{1}{6} & \frac{2}{3}
   \end{pmatrix}. \]
This yields:
\[ c_1
   = \pi^{-4} \left[ 162 - 99 \sqrt{3} \pi + \frac{99 \pi^{2}}{2}
     - \frac{5 \sqrt{3} \pi^3}{12} - \frac{25 \pi^{4}}{36} \right], \]
and
\[ c_\cG 
   = \frac{c_0 c_1}{\pi} 
   = \frac{\sqrt{3}}{\pi^5} \left[ 162 - 99 \sqrt{3} \pi + \frac{99 \pi^{2}}{2}
     - \frac{5 \sqrt{3} \pi^3}{12} - \frac{25 \pi^{4}}{36} \right]. \]

\subsection{Reflection and 180-degree symmetry}
\label{ssec:180}

In this section we consider a graph that is symmetric under reflection in the $x$- and $y$-axes,
and hence also under rotation by 180 degrees. The function $f_U$ we get is different from the
conformal radius, even though the potential kernel is asymptotically rotation symmetric.

We take the vertex set to be
\eqnst{
  \cV 
  = \{ n e + m f : \text{$n, m \in \Z$} \}, \quad \text{with } e = \mathbf{1}, \quad
    f = \sqrt{2} \mathbf{i}; }
with each horizontal pair of nearest neighbours connected by $2$ parallel edges, and vertical pairs
by a single edge. The vertical stretching by $\sqrt{2}$ ensures isotropy of the random walk covariance matrix.
Local values of the potential kernel can be computed explicitly via the Fourier transform 
(see \cite{Sp-book,KW16}). 
We get the potential kernel values (see Appendix \ref{ssec:xy-lattice} for more details):
\begin{align*}
  A(o,e) &= \frac{\arcsin(\sqrt{6}/3)}{2 \pi} &
  A(o,f) &= \frac{\arccos(\sqrt{6}/3)}{\pi} & 
  A(e,f) &= \frac{\sqrt{2}}{2 \pi} \\
  A(e,-e) &= \frac{3 \arcsin(\sqrt{6}/3) - \sqrt{2}}{2\pi} & 
  A(f,-f) &= \frac{6 \arccos(\sqrt{6}/3) - 2\sqrt{2}}{\pi} &
\end{align*}
From these and symmetry the relevant entries of the transfer current matrix are readily obtained.

It is convenient to collapse the pairs of parallel edges into single edges with conductance 
$c(e) = 2$ and putting $c(f) = 1$ (resulting in smaller matrices). Denote by $p(e)$ and $p(f)$,
respectively, the probabilities that in the weighted uniform spanning tree on the full lattice 
$o$ is a leaf with the edge at $o$ equal to $e$ and $f$, respectively. Then the probabilities 
in the original graph that $o$ is a leaf with a horizontal or vertical edge, are 
$p(e)/c(e) = p(e)/2$ and $p(f)/c(f) = p(f)$, respectively. We similarly denote 
$p_{U_\eps}(e)$ and $p_{U_\eps}(f)$ the respective weighted spanning tree probabilities 
in $U_\eps$. We get the $\eps^2$ order correction term in $p_{U_\eps}(e)/c(e)$
\eqnsplst{
 &\eps^2 J_{1,1} 
   \frac{-12 \arcsin\! \left(\frac{\sqrt{6}}{3}\right)^{2}+16 \arcsin\! \left(\frac{\sqrt{6}}{3}\right) \pi -4 \pi^{2}-8+16 \sqrt{2}\, \arccos\! \left(\frac{\sqrt{6}}{3}\right)}{\pi^{2}} \\
 &\qquad\quad + \eps^2 J_{2,2} \frac{6 \arcsin\! \left(\frac{\sqrt{6}}{3}\right)^{2}+\left(-8 \pi +4 \sqrt{2}\right) \arcsin\! \left(\frac{\sqrt{6}}{3}\right)+2 \pi^{2}-4}{\pi^{2}}, }
with $p_{U_\eps}(f)/c(f)$ having the same order $\eps^2$ terms. We see that the
the coefficients of $I_{1,2}$ and $I_{2,1}$ vanish.
The ratio of the coefficients of $I_{1,1}$ and $I_{2,2}$ is $\approx 1.347535564$.
This shows that $\eps^2$ order term is not conformally covariant with the asymptotically 
rotation invariant embedding.

The stretch in the vertical direction required to make the coefficients of
$J_{1,1}$ and $J_{2,2}$ equal is:
\eqnsplst{ 
   b 
   &= \frac{2 \sqrt{-3 \arcsin\! \left(\frac{\sqrt{6}}{3}\right)^{2}+4 \arcsin\! \left(\frac{\sqrt{6}}{3}\right) \pi -\pi^{2}+4 \sqrt{2}\, \arccos\! \left(\frac{\sqrt{6}}{3}\right)-2}}{\sqrt{3 \arcsin\! \left(\frac{\sqrt{6}}{3}\right)^{2}+\left(-4 \pi +2 \sqrt{2}\right) \arcsin\! \left(\frac{\sqrt{6}}{3}\right)+\pi^{2}-2}} \\
   &\approx  1.641667175 
   > \sqrt{2}. }
The same stretch $b$ makes the two-point correlation function between heights $0$
(computable as in \cite{Durre2009}), also conformally covariant, hinting at a possible
generalisation to other lattices than we considered. We have no intuitive understanding
of the unusual scaling $b$ required.

\subsection{An isoradial graph}
\label{ssec:isorad}

The following example shows that on general isoradial graphs \cite{K02} even the coefficients of 
$J_{1,2}$ and $J_{2,1}$ will not vanish. Consider a planar isoradial graph 
where $o$ has the three neighbours: $e, f, g$, with corresponding rombus half-angles
$\theta_e = \pi/3$, $\theta_f = \pi/4$, $\theta_g = 5\pi/12$, respectively. 
That is, taking the rombus edges to have length $1$, we 
can set (see Figure \ref{fig:isorad} in Appendix \ref{ssec:isorad-calc}):
\eqnst{
  e:= 1 \qquad\qquad f := \frac{1+i}{2} + \sqrt{3} \frac{-1+i}{2} \qquad\qquad
    g:= (1-\sqrt{3}) \frac{1+i}{2}, }
with conductances 
\eqnst{
  c(e) = \tan(\pi/3) = \sqrt{3} \qquad 
  c(f) = \tan(\pi/4) = 1 \qquad 
  c(g) = \tan(5 \pi/12) = \frac{\sqrt{3}+1}{\sqrt{3}-1}. }
The total conductance of $o$ equals $C := c(e) + c(f) + c(g)$.
We extend the graph in an arbitrary way that is isoradial and doubly periodic
(see Figure \ref{fig:isorad-periodic} in Appendix \ref{ssec:isorad-calc} for a 
possible periodic extension).
Using the method of \cite[Theorem7.1]{K02}, we obtain the potential kernel values (see 
Appendix \ref{ssec:isorad-calc} for more details):
\begin{align*}
  A(o,e) &= \frac{\sqrt{3}}{9} \qquad&\qquad 
  A(o,f) &= \frac{1}{4} \qquad&\qquad 
  A(o,g) &= \frac{5}{12} \frac{\sqrt{3}-1}{\sqrt{3}+1} \\
  A(e,f) &= \frac{1}{12} \frac{\sqrt{3}+1}{\sqrt{3}-1} \qquad&\qquad
  A(e,g) &= \frac{1}{4} \qquad&\qquad 
  A(f,g) &= \frac{\sqrt{3}}{6}. 
\end{align*}
From these the relevant entries of the transfer current matrix are readily deduced. 

As in Section \ref{ssec:180}, we let $p(e), p(f), p(g)$ denote the respective probabilities
that in the weighted spanning tree on the full graph $o$ is a leaf with $e, f, g$ present,
respectively. As a proxy to the height $0$ probability of the sandpile (interpretable as a
limit of probabilities in graphs with rational conductances), we get
\[ \frac{p(e)}{c(e)/C}
   = \frac{p(f)}{c(f)/C}
   = \frac{p(g)}{c(g)/C}
   = -\frac{1}{36}+\frac{11 \sqrt{3}}{72}. \]
For the order $\eps^2$ correction to $p_{U_\eps}(e)/(c(e)/C)$ we get
\[  \eps^2 C \left[ J_{1,1} \left( \frac{\sqrt{3}}{2}-\frac{2}{3} \right) 
     + J_{2,2} \left( \frac{1}{6} \right) 
     + \left( J_{1,2} + J_{2,1} \right) \left( \frac{\sqrt{3}}{6}-\frac{1}{4} \right) \right], \]
with the corrections for $p_{U_\eps}(f)/(c(f)/C)$ and $p_{U_\eps}(g)/(c(g)/C$
being identical.

\textbf{Acknowledgements}
The figures and some of the computations in this paper were carried out using Maple™.

\section{Appendix}
\label{sec:appendix}

Here we present a direct proof of Proposition \ref{prop:A-asymp} based on the 
result of \cite{FU96}. We start with a summary of some preliminary results in 
the next section.

\subsection{Martingales}
\label{ssec:mart}

We will need the following simple lemma.

\begin{lemma}
\label{lem:mart}
For any function $g : \cV \to \R$, and for any $v \in \cV$, we have that
\eqnst{
  g(S_n) - g(S_0) - \sum_{t=0}^{n-1} \frac{1}{\deg(S_t)} \Delta g (S_t), \quad n \ge 0, }
is a martingale under $\Pr^v$.
\end{lemma}

In applying Lemma \ref{lem:mart}, we will make use of the following simple result.

\begin{lemma}
\label{lem:lapl}
Let $g : \R^2 \to \R$ have continuous and bounded partial derivatives of order up to $3$.
Then for any $u \in \cV$ we have that
\eqnst{
  \Delta g(u)
  = \sum_{u' \sim u} (g(u) - g(u'))
  = - \frac{1}{4} \Delta^{(c)} g (u) \sum_{u' \sim u} |u - u'|^2
    + O( \| D^3 g \|_\infty \max_{u'\sim u} |u'-u|^3 ), }
where $\Delta^{(c)} g = (\partial_{11} + \partial_{22}) g$ is the (continuous) Laplacian.
\end{lemma}

\begin{proof}
By a Taylor expansion, we have
\eqnsplst{
  g(u) - g(u')
  = \nabla g \cdot (u - u') - \frac{1}{2} (D^2 g)(u-u',u-u') + O( \| D^3 g \|_\infty 
  \max_{u'\sim u} |u'-u|^3 ). }
Due to the rotational symmetry assumed, the sum over $u' \sim u$ of the first term on the right hand side 
vanishes. For the second term, fix some $u'_0 \sim u$ and let $O$ denote rotation by 
an angle $2\pi/q$, where $q = q_u \ge 3$ is the symmetry angle of $\cG$ about $u$. 
Let $u'_r - u = O^r (u'_0 - u)$.
Abbreviate $\cos(2\pi r/q_u) =: c(r)$ and $\sin(2\pi r/q_u) =: s(r)$, and 
$\partial_{ij} = \partial_{ij} g$. Then we have
\eqnspl{e:D2g}{
  &D^2 g (u'_r - u, u'_r - u)
  = (u'_r - u)^T \begin{pmatrix}
      \partial_{11} g & \partial_{12} g \\
      \partial_{21} g & \partial_{22} g
  \end{pmatrix} (u'_r - u) \\
  &= (u'_0 - u)^T \begin{pmatrix}
      c(r) & s(r) \\
      -s(r) & c(r)
  \end{pmatrix} \begin{pmatrix}
      \partial_{11} & \partial_{12} \\
      \partial_{12} & \partial_{22}
  \end{pmatrix} \begin{pmatrix}
      c(r) & -s(r) \\
      s(r) & c(r)
  \end{pmatrix} (u'_0 - u) \\
  &= (u'_0 - u)^T \\
  &\quad \begin{pmatrix}
      c^2(r) \partial_{11} + s^2(r) \partial_{22} + 2 s(r) c(r) \partial_{12} 
            & \!\!-c(r) s(r) [ \partial_{11} - \partial_{22} ] + (c^2(r) - s^2(r)) \partial_{12} \\
      -c(r) s(r) [ \partial_{11} - \partial_{22} ] + (c^2(r) - s^2(r)) \partial_{12}
            & \!\!s^2(r) \partial_{11} + c^2(r) \partial_{22} - 2 s(r) c(r) \partial_{12}
  \end{pmatrix} \\
     &\qquad (u'_0 - u). }
Using the identities 
\begin{align}
\label{e:averaging}
  \sum_{r=0}^{q-1} c^2(r)
  &= \frac{q}{2} & 
  \sum_{r=0}^{q-1} s^2(r)
  &=\frac{q}{2} \\
  \sum_{r=0}^{q-1} c(2r)
  &= 0 & 
  \sum_{r=0}^{q-1} s(2r)
  &= 0, 
\end{align}
the sum of \eqref{e:D2g} over $r = 0, \dots q-1$ equals
\eqnsplst{
  (u'_0 - u)^T \begin{pmatrix}
      \frac{q}{2} \Delta^{(c)} g & 0 \\
      0 & \frac{q}{2} \Delta^{(c)} g
  \end{pmatrix} (u'_0 - u)
  = \frac{q}{2} \Delta^{(c)} g |u'_0 - u|^2
  = \frac{1}{2} \Delta^{(c)} g \sum_{r=0}^{q-1} |u'_r - u|^2. }
Summing over all the orbits of neighbours of $u$ (under the action of $O$) gives the claim.
\end{proof}

Sometimes we will need to consider $\Z^2$-isomorphic sub-lattices of $\cG$ of the following type:
\eqnst{
  \mathcal{L}^w 
  := \{ w + n h_1 + m h_2 : n, m \in \mathbb{Z} \}, }
where $w \in \cV$, and $h_1, h_2$ are fixed $\mathbb{Z}^2$-periodicity vectors of $\cG$. 
Let $(S_t)_{t \ge 0}$ denote discrete time simple random walk on $\cG$, with 
filtration $(\cF_t)_{t \ge 0}$. Let 
\eqnsplst{
  \tau_0
  &= 0; \\
  \tau_i
  &= \inf \{ t > \tau_{i-1} : S_t \in \mathcal{L}^w \} \quad \text{for $i \ge 1$.} }
Put $Y_i = S_{\tau_i}$, $i \ge 0$.

The following lemma is easy to see by regarding the walk $S$ modulo the lattice $\mathcal{L}^w$
as a finite state Markov chain.

\begin{lemma}
\label{lem:tau-decay}
There exists a constant $c > 0$ only depending on $\cG$ and the fixed vectors $h_1, h_2$ such 
that for all $x, w \in \cV$ we have
\eqn{e:tau-decay}{
  \mathbf{P}^x [ \tau_1 > t]
  \le e^{- c t}. }
\end{lemma}

\begin{lemma}
\label{lem:Y-uncorr}
We have that under $\mathbf{P}^w$ the random variables $\Re (Y_1-Y_0)$ and $\Im (Y_1-Y_0)$ are uncorrelated, 
therefore under $\mathbf{P}^w$ the random walk $Y_i - Y_0$, $i \ge 0$ has covariance matrix a
constant times the identity.
\end{lemma}

\begin{proof}
Due to the rotational symmetry around every vertex of $\cG$, we have that
\eqnspl{e:martSw}{
  \E^w [ S_{t+1} - S_t \,|\, \cF_t ] = 0. }
By the same reason, we also have
\eqnsplst{
  \E^w [ \Re(S_{t+1} - S_t) \Im(S_{t+1} - S_{t}) \,|\, \cF_t ]
  = 0. }
This implies that
\eqnsplst{
  &\E^w [ \Re(S_{t+1} - S_0) \Im(S_{t+1} - S_0) \,|\, \cF_t ] \\
  &\quad = \Re(S_{t} - S_0) \Im(S_{t} - S_0) 
    + \E^w [ \Re(S_{t+1} - S_t) \Im(S_{t} - S_0) \,|\, \cF_t ] \\
  &\qquad\quad + \E^w [ \Re(S_{t} - S_0) \Im(S_{t+1} - S_t) \,|\, \cF_t ] \\
  &\quad = \Re(S_{t} - S_0) \Im(S_{t} - S_0), }
where we used \eqref{e:martSw}. By the Optional Stopping Theorem, using the decay in Lemma 
\ref{lem:tau-decay},
we get that
\eqnst{
  \E^w [ \Re(S_{\tau_1} - S_0) \Im(S_{\tau_1} - S_0) ]
  = 0, }
as claimed.
\end{proof}

\subsection{{Alternative proof of Proposition \ref{prop:A-asymp}}}
\label{ssec:Proof-prop:A-asymp}

Let $A^{(0)}_w(v)$, $v \in \cL^w$, denote the potential kernel of the random walk $Y$, that is,
\eqnsplst{
  A^{(0)}_w(v)
  = \lim_{N \to \infty} \, \frac{1}{\deg(w)} \sum_{i=0}^{N} \left[ \Pr^w[ Y_i = w ] 
    - \Pr^w[ Y_i = v ] \right]. }
Applying \cite[Theorem 2]{FU96} to the walk $Y$, in light of Lemma \ref{lem:Y-uncorr}, 
we get that for some constants $c_0([w]) = c_0([w],\cG)$, $C_0([w]) = C_0([w],\cG)$ 
and a smooth function $U_2$ on the unit circle (also depending on $[w]$ and $\cG$), we have 
\eqn{e:A0-asymp}{
  A^{(0)}_w(v)
  = \frac{c_0([w])}{\pi} \log |v - w| + C_0([w]) + \frac{U_2((v - w)/|v - w|)}{|v - w|^2} 
    + O \left( |v - w|^{-3} \right), }
as $| v - w | \to \infty$, $v \in \cL^w$.

We will need the following a-priori bound:
\eqn{e:A-a-priori}
{ A(w,v) 
  = O \left( 1 \vee \log | v - w | \right), \quad v, w \in \cV. }

The following lemma shows that $A^{(0)}_w(\cdot)$ coincides 
with the restriction of $A(w,\cdot)$ to the subset $\cL^w$. 
A proof can be found in \cite[Lemma 21]{HS21}.

\begin{lemma}[{\cite[Lemma 21]{HS21}}] 
For all $w \in \cV$, $v \in \cL^w$ we have $A(w,v) = A^{(0)}_w(v)$.
\end{lemma}

Thus, in order to prove Proposition \ref{prop:A-asymp}, it remains to extend the 
asymptotics \eqref{e:A0-asymp} from the subset $\cL^w$ to all of $\cV$. 

\begin{proof}[Proof of Proposition \ref{prop:A-asymp}.]
Let us fix $w \in \cV$. Let $v \in \cV$ with $|v - w| \gg 1$. Consider the stopping times
\eqnst{
  \sigma
  := \inf \{ t \ge 0 : | S_t - v | > |v-w|/2 \}; \qquad\qquad
  \tau'
  := \tau_1 \wedge \sigma. }
Using that $A(w,\cdot)$ is discrete harmonic away from $w$, and using the Optional 
Stopping Theorem, for any $N \ge 1$ we have that
\eqnst{
  A(w,v)
  = \E^v \left[ A(w,S_{\tau' \wedge N}) \right]
  = \E^v \left[ A(w,S_{\tau'}) \mathbf{1}_{\tau' \le N} \right]
    + \E^v \left[ A(w,S_{N}) \mathbf{1}_{\tau' > N} \right]. }
The upper bound of \eqref{e:A-a-priori} and the decay bound
\eqref{e:tau-decay} imply that the second term is exponentially small in $N$. 
Hence by \eqref{e:A-nonneg} and monotone convergence we get
\eqnst{
  A(w,v)
  = \E^v \left[ A(w,S_{\tau'}) \right]. }
Using that 
\eqnst{
  \Pr^v [ |S_{\tau_1} - v| \ge |v - w|/2 ]
  \le \Pr^v [ \tau_1 > c(\cG) |v - w|]
  \le e^{- c |v - w|}, }
we have that
\eqnst{
  A(w,v)
  = \E^v \left[ A(w,S_{\tau_1}) \mathbf{1}_{\tau_1 \le \sigma} \right]
    + O \left( e^{- c |v - w|} \right). }
On the event $\{ \tau_1 \le \sigma \}$ we have by \eqref{e:A0-asymp} that
\eqnst{
  A(w,S_{\tau_1})
  = \frac{c_0([w])}{\pi} \log |S_{\tau_1} - w| + C_0([w]) 
    + \frac{U_2((S_{\tau_1} - w)/|S_{\tau_1} - w|)}{|S_{\tau_1} - w|^2} 
    + O \left( |v - w|^{-3} \right). }
Applying Lemma \ref{lem:mart} twice, once with $g_1(u) = \log| u - w |$ and then with 
$g_2(u) = \frac{U_2((u - w)/|u - w|)}{|u - w|^2}$ we have 
\eqnspl{e:apply-mart}{
  &\E^v \left[ A(w,S_{\tau_1}) \mathbf{1}_{\tau_1 \le \sigma} \right] \\
  &\quad = \frac{c_0([w])}{\pi} \log|v - w| + C_0([w]) 
     + \frac{U_2((v - w)/|v - w|)}{|v - w|^2} + O \left( |v - w|^{-3} \right) \\
  &\qquad\quad + \E^v \left[ \sum_{t=0}^{\tau_1-1} \frac{1}{\deg(S_t)} \Delta g_1(S_t) 
      \mathbf{1}_{\tau_1 \le \sigma} \right]
    + \E^v \left[ \sum_{t=0}^{\tau_1-1} \frac{1}{\deg(S_t)} \Delta g_2(S_t) 
      \mathbf{1}_{\tau_1 \le \sigma} \right]. }
We bound from above the two sums on the right hand side. For $g_1$, consider the function
$\tilde{g}_1 : \R^2 \to \R$ defined by the same formula as $g_1$, but made to vanish outside 
the ball $B(v, 3|v-w|/4)$. Applying Lemma \ref{lem:lapl} to $g_1$ gives 
\eqnst{
  \Delta g_1(u)
  = \frac{1}{4} \Delta^{(c)} \tilde{g}_1 \sum_{u' \sim u} |u' - u|^2 + O ( |v - w|^{-3} )
  = O ( |v - w|^{-3} ), }
where $\Delta^{(c)}$ is the (continuous) Laplacian in $\R^2$. This gives that the first
expectation on the right hand side of \eqref{e:apply-mart} is bounded in absolute value by
\eqnst{
  \frac{C}{|v-w|^3} \E^v [ \tau_1 ]
  = O ( |v - w|^{-3} ). }
For $g_2$ we have $\Delta^{(c)} g_2 = O(|v-w|^{-4})$ so Lemma \ref{lem:lapl} gives 
\eqnst{
  \Delta g_2(u)
  = O ( |v - w|^{-4} ). }
This gives that the second expectation on the right hand side of \eqref{e:apply-mart} is 
$O ( |v - w|^{-4} )$.

In order to complete the proof, we need to show that $c_0$, $C_0$, $U_2$ in fact do not depend 
on $[w]$, and identify the constant in the leading term. From $A(w,v) = A(v,w)$, and the
form of the leading term, we have that $c_0([w]) = c_0([v])$, and hence $c_0$ does not depend
on $[w]$. Subtracting the leading term from $A(w,v) = A(v,w)$, we see that $C_0$ does not
depend on $[w]$ either. Now the same follows for the function $U_2$. To identify $c_0$, 
we transform the walk $Y$ to a walk on $\mathbb{Z}^2$ by setting 
$X_n = M^{-1} (Y_n - Y_0)$, where the matrix $M$ has columns $h_1, h_2$. Thus the 
covariance matrix of $X_1$ equals
\eqnsplst{ 
   Q
   &:= \E^w [ M^{-1} (Y_1-Y_0) (Y_1 - Y_0)^T (M^{-1})^T ] 
   = \E^w [ |Y_1 - Y_0|^2 ] M^{-1} (M^{-1})^T \\
   &=: \frac{\sigma_1^2}{2} M^{-1} (M^{-1})^T. }
The determinant of $\sqrt{Q}$ hence equals $\frac{\sigma_1^2}{2 |\det(M)|}$, and 
therefore, from \cite[Theorem 1]{FU96}, we have
\[ c_0
   = \frac{2 |\det(M)|}{\sigma_1^2} \frac{1}{\deg(w)}. \]
To determine $\sigma_1^2$, using the law of large numbers, we have
\eqnsplst{ 
   \sigma_1^2
   &= \lim_{n \to \infty} \E^w \left[ \frac{|S_{\tau_n} - S_0|^2}{n} \right]
   = \lim_{n \to \infty} 
     \E^w \left[ \frac{|S_{\tau_n} - S_0|^2}{\tau_n} \frac{\tau_n}{n} \right] \\
   &= \E^w [ \tau_1 ] \, \lim_{m \to \infty} \E^w \left[ \frac{|S_m - S_0|^2}{m} \right]. }
Writing $T$ for a fundamental region of $\cL^w$, and using the ergodicity of the 
finite state Markov chain on $T$ induced by the walk $S$, we have that 
\[ \lim_{m \to \infty} \E^w \left[ \frac{|S_m - S_0|^2}{m} \right]
   = \sum_{u \in T} \mu(u) \E^u [ |S_1 - S_0|^2 ]
   = \frac{1}{Z_T} \sum_{u \in T} \sum_{y \sim u} |y-u|^2, \]
where $\mu$ is the invariant distribution of the Markov chain on $T$. On the other hand,
\[ \E^w [ \tau_1 ]
   = \frac{1}{\mu(w)}. \]
this gives
\[ c_0
   = 2 |\det(M)| \frac{\mu(w)}{\deg(w)} 
    \frac{Z_T}{\sum_{u \in T} \sum_{y \sim u} |y - u|^2}
   = \frac{2 |\det(M)|}{\sum_{u \in T} \sum_{y \sim u} |y - u|^2}. \]
\end{proof}

\section{Appendix}
\label{sec:proof-harmonic-conv}

Here we give a summary of the proof of Theorem \ref{thm2_3}. In Section \ref{ssec:Beurling}
we explain how a Beurling-type estimate can be proved in our setting, which is then used 
in Section \ref{ssec:convergence-proof} in the convergence proof.

\subsection{A Beurling-type estimate}
\label{ssec:Beurling}

Let $z \in U_\eps$ and assume that there exists $e \in \partial U_\eps$ with $e_- = z$.
Let $F_R = \partial U_\eps \cap B(z,R)$. Let $\tau_{F_R}$ denote the first time the 
simple random walk on $\cG$ has crossed an edge belonging to $F_R$. Let 
\[ \xi_{z,r}
   = \inf \{ n \ge 0 : S_n \not\in B(z,r) \}. \]
We will need the statement that there exists $C$ only dependent on $\cG$, such that 
\begin{equation}
\label{e:Beurling-edge}
   \Pr^z [ \xi_{z,2R} < \tau_{F_R} ]
   \le C \left( \frac{\eps}{R} \right)^{1/2}. 
\end{equation} 
See \cite{Kesten-Beurling,Lawler-book-intersections,LL04} for close variants, where 
$\tau_{F_R}$ is replaced by the hitting time of a set of vertices (rather than edges).
The proof of \eqref{e:Beurling-edge} can easily be reduced to the proof of a vertex-version. Indeed, let $F_R^- := \{ f_- : f \in F_R \}$, let 
$\tau_{F_R^-} = \inf \{ n \ge 0 : S(n) \in F_R^- \}$ and 
\[ M(R) 
   = \max_{f \in F_R} \Pr^{f_-} [ \xi_{f_-,2R} < \tau_{F_R} ], \qquad
\text{ and } \qquad
   M_-(R) 
   = \max_{f \in F_R} \Pr^{f_-} [ \xi_{f_-,2R} < \tau_{F_R^-} ], \]
Assume $M_-(R) \le C (\eps/R)^{1/2}$. Using that starting at $f_-$, the edge $f$ 
cannot be crossed immediately, we have
\begin{equation*}
\begin{split}
    &\Pr^{f_-} [ \xi_{f_-,2R} < \tau_{F_R} ] 
    \le \sum_{w \sim f_- : w \not= f_+} 
        \Pr^{f_-} [ S(1) = w ] \\
    &\quad\quad \times \left[ \Pr^w [ \xi_{w,2R} < \tau_{F_R^-} ] + 
        \sum_{g \in F_R} \Pr^w [ \tau_{F_R^-} \le \xi_{w,2R},\, S(\tau_{F_R^-}) = g_- ]
        \Pr^{g_-} [ \xi_{g_-,2R} < \tau_{F_R} ] \right] \\
    &\quad \le \left( 1 - \frac{1}{D} \right) \left( C (\eps/R)^{1/2} + M(R) \right), 
\end{split}
\end{equation*}
where $D$ is the maximum degree. This yields $M(R) \le (D-1) C (\eps/R)^{1/2}$.

For the proof of the bound on $M_-(R)$, the line of argument leading to the proof of
\cite[Theorem 2.5.2]{Lawler-book-intersections} ($d=2$ case) can be followed. The argument 
requires the following ingredients.
\begin{itemize}
\item[(i)] Asymptotics of the potential kernel; see Proposition \ref{prop:A-asymp}.
\item[(ii)] Escape probability from a line for a killed random walk. For this, the proof 
 of \cite[Proposition 3.1]{Fukai-halfline} (that considers random walks on $\mathbb{Z}^2$)
 can be adapted. Indeed, assume, for 
 simplicity, that the walk starts at $o$, and recall 
 $\tau_1 =  \inf \{ n \ge 1 : S(n) \in \cL^o \}$, where $\cL^o$ is the lattice generated
 by $\mathbf{h_1,h_2}$. Let $L = \{ n \mathbf{h_1} : n \in \mathbb{Z} \}$, and let
 \begin{equation}
 \label{e:G_L-Fourier}
    G_L(\lambda)
    = \sum_{n=0}^\infty \lambda^n \, \Pr^o [ S(n) \in L ]
    = \frac{1}{2 \pi} \int_{-\pi}^\pi \frac{1}{1 - \phi_\lambda(\theta)} \, d\theta,
 \end{equation}  
 where $\phi_\lambda(\theta) = \E^o [ \lambda^{\tau_1} e^{i \theta S^{(2)}(\tau_1)} ]$, with
 $S(m) = (S^{(1)}(m), S^{(2)}(m))$. We have the Taylor expansion: 
 \[ \phi_\lambda(\theta)
    = \E^o [ \lambda^{\tau_1} ] 
      - \frac{\theta^2}{2} \E^o [ \lambda^{\tau_1} S^{(2)}(\tau_1)^2 ] + e(\theta,\lambda)
    =: \alpha(\lambda) - \frac{\theta^2}{2} \beta(\lambda) + e(\theta,\lambda). \]
 Hence, the integral in the right hand side of \eqref{e:G_L-Fourier} yields
 \eqnsplst{ 
    \mathrm{const} \cdot \left( \frac{1 - \alpha(\lambda)}{\beta(\lambda)/2} \right)^{-1/2} 
    &\sim \mathrm{const} \cdot (1 - \lambda)^{-1/2} \left( \frac{2 \E^o [ \tau_1 ]}{\E^o [ S^{(2)}(\tau_1)^2 ]} \right)^{-1/2} \\
    &= \mathrm{const} \cdot (1 - \lambda)^{-1/2}. }
 This gives $\Pr^o [ T_\lambda < \tau_L ] \sim \mathrm{const} \cdot (1 - \lambda)^{1/2}$, 
 where $\tau_L = \inf \{ n \ge 1 : S(n) \in L \}$ and $T_\lambda$ is a $\mathsf{Geom}(1 - \lambda)$
 killing time independent of the walk.
\item[(iii)] Since the walk $Y_k = S(\tau_k)$, $k = 0, 1, 2, \dots$ is $180$-degree symmetric,
 it can be deduced from (ii), as in \cite{Lawler-book-intersections,Fukai-halfline}, that 
 $\Pr^o [ T_\lambda < \tau_{L_+} ] \sim \mathrm{const} \cdot (1 - \lambda)^{1/4}$.
\item[(iv)] The remaining adaptations to the proof in \cite{Lawler-book-intersections} 
 are minor. Note that the assumption made about boundary points
 $a \in \partial_\infty U$ in Assumption \ref{a:U}, implies the required 
 `discrete connectivity' property of $\partial U_\eps$.
\end{itemize}

\subsection{Convergence proof}
\label{ssec:convergence-proof}

We define
\eqnsplst{
  \Delta^{(\eps)} F(z)
  &= \sum_{f \in \eps \cE: f_- = z} (F(z) - F(z+f)), \quad z \in \eps\cV,\, F : \eps \cV \to \mathbb{R}, \\
  \rho^2_{z,\eps}
  &= \sum_{f \in \eps\cE: f_- = x} |f|^2, \quad z \in \eps\cV. }

\begin{proof}[Summary of proof of Theorem \ref{thm2_3}.] \ \\ 
(i) One can prove that the $u_\eps$ are equicontinuous by the arguments 
given in \cite[Appendix]{CS11}, that go back to \cite{D53,B08}. Since the 
focus in \cite{CS11} is isoradial graphs, we give the main steps to 
indicate how they extend to our setup.

\emph{Step 1.} It is easy to verify the following discrete version of Green's formula:
\eqnsplst{
  \sum_{z \in U_\eps} \left[ H \Delta^{(\eps)} F - F \Delta^{(\eps)} H \right](z) 
  = - \sum_{f \in \partial U_\eps}
    \left[ H(f_-) G(f_+) - H(f_+) G(f_-) \right]. }

\emph{Step 2.} Let $B_{R,\eps}(v)$ be the discrete ball of (Euclidean) radius 
$R$ around $v \in \eps\cV$. It can be derived from the potential kernel asymptotics 
in Proposition \ref{prop:A-asymp} that
\eqnst{
  H_{B_{R,\eps}}(v,f) 
  \asymp \frac{\eps}{R} \quad \text{for all $f \in \partial B_{R,\eps}(v)$,} }
where the implicit constants only depend on $\cG$. See the proof 
of \cite[Proposition A.1]{CS11}.

\emph{Step 3.} Now from Steps 1 and 2 follows the mean value property 
(see \cite[Proposition A.2]{CS11}): 
if $H : B_{R,\eps}(v) \cup \partial B_{R,\eps}(v) \to \mathbb{R}$ is non-negative 
and discrete harmonic in $B_{R,\eps}(v)$, then
\eqnsplst{
 \left| H(v) - \frac{c_0}{2 \pi R^2} \sum_{z \in B_{R,\eps}(v)} H(z) \rho^2_{z,\eps} \right|
 \le C \frac{\eps}{R} H(v), }
where we denote $c_0 = c_0(\cG) = \det(M)/(Z_T \sigma^2)$. Indeed, let
\[ F(z)
   = - A(\eps^{-1} z, \eps^{-1} v) + \frac{c_0}{\pi} \log (R/\eps) + C_0 
     - \frac{c_0}{2\pi} \frac{|z-v|^2 - R^2}{R^2}, \quad z\in B_{R,\eps}. \]
Due to Lemma \ref{lem:lapl} and Proposition \ref{prop:A-asymp} we have 
\eqnsplst{
  \Delta^{(\eps)} F (z)
  &= - \delta_{z,v} + \frac{c_0}{2 \pi R^2} \rho^2_{z,\eps}, \quad z \in B_{R,\eps} \\
  F(z) 
  &= \mathcal{O}(\eps^2/R^2) \quad \text{when $|z-v| = R + \mathcal{O}(\eps)$}. }
Thus applying Green's formula to $F \pm C (\eps^2/R^2)$ and $H$ we deduce
\eqnsplst{
 -C \frac{\eps^2}{R^2} \sum_{f \in \partial B_{R,\eps}} H(f) 
 \le H(v) - \frac{c_0}{2 \pi R^2} \sum_{z \in B_{R,\eps}(v)} H(z) \rho^2_{z,\eps}
 \le C \frac{\eps^2}{R^2} \sum_{w \in \partial^{\mathrm{int}} B_{R,\eps}} H(w), }
where $\partial^{\mathrm{int}} B_{R,\eps}
= \{ w \in B_{R,\eps} : \text{$\exists f \in \partial B_{R,\eps}$ such that $w = f_-$}\}$.
Then Step 2 and Theorem \ref{thm2_2} imply the claim.

\emph{Step 4.} As in \cite[Proposition 2.7]{CS11} and \cite[Corollary 2.9]{CS11}, it follows from 
Step 3 that if $H : B_{R,\eps}(v) \cup \partial B_{R,\eps}(v) \to \mathbb{R}$ is discrete harmonic 
in $B_{R,\eps}(v)$ (not necessarily non-negative), and $v' \sim v$, then
\[ | H(v) - H(v') |
     \le C \frac{\eps}{R} \left[ \max_{y \in \partial B_{R,\eps}(v)} |H(y)| \right]. \]
In particular, if the boundary condition satisfies $\| g \|_\infty \le M$, then $u_\eps$
are uniformly bounded and locally equicontinuous with a modulus of continuity only dependent
on $M$ and the distance from the boundary.

\emph{Step 5.} 
By Step 4, subsequential limits of $(u_\eps)$ exist. Let $u$ be any such limit. 
From Step 3 we get that $u$ satisfies the (continuous) mean-value property. 
Indeed, after adding a suitable constant $c_u$ to $u_\eps$
to make $\tilde{u}_\eps = u_\eps + c_u$ non-negative in the ball $B_{R,\eps}(v)$, we have
\eqnsplst{
  \frac{c_0}{2 \pi R^2} \sum_{z \in B_{R,\eps}(v)} \tilde{u}_\eps(z) \rho^2_{z,\eps}
  &= \frac{c_0}{2 \pi R^2} \sum_{\substack{z \in B_{R,\eps}(v) \\ [z] = [o]}} \tilde{u}_\eps(z)
     \sum_{t \in T} \rho^2_{t,\eps} + \mathcal{O}( \eps/R ) \\
  &= \frac{\eps^2 |\det(M)|}{\pi R^2} \sum_{\substack{z \in B_{R,\eps}(v) \\ [z] = [o]}} \tilde{u}_\eps(z)
     + \mathcal{O}( \eps/R ). }
The last expression converges, as $\eps \to 0$ along the chosen subsequence, to 
\[ \frac{1}{\pi R^2} \int_{z : |z - v| \le R} \tilde{u}(z) m(dz), \]
where $\tilde{u} = u + c_u$, and $m$ is Lebesgue measure. Therefore we deduce 
that $u$ is harmonic.

\emph{Step 6.} 
One checks that $u$ satisfies the boundary conditions of the continuous Dirichlet problem, 
by the arguments of \cite[Proposition 3.3]{CS11}. This uses the Beurling-type estimate
of Section \ref{ssec:Beurling}. 

\medbreak

(ii) Let $h_1, h_2 \in \mathbb{R}^2$ be two linearly independent vectors such that $\cV$ is invariant under 
translations by $h_1, h_2$. The functions
\[ v^{h_i}_\eps(x)
   = \frac{1}{\eps} (u_\eps(x + \eps h_i) - u_\eps(x)), \quad i = 1, 2 \]
are discrete harmonic at every $x$ such that both $x$ and $x+h_i$ are in $U_\eps$. The arguments 
in part (i) can be applied to see that $v^{h_i}_\eps$ are locally equicontinuous, and hence along 
a subsequence converge to a limit $v^{h_i}$ locally uniformly in $U$. From this convergence, we 
deduce that 
$u(x+\delta h_i) - u(x) = \int_0^\delta v^{h_i}(s) ds$ for any $x$ and $\delta$ such that the line segment
joining $x$ and $x + \delta h_i$ lies in $U$. This implies that 
$v^{h_i}(x) = \nabla u(x) \cdot h_i$ for $i = 1,2$, throughout $U$. Consequently, the convergence holds
without restricting to a subsequence.

In order to pass to the corresponding statement for displacements along edges of the graph, assume that 
$x, x+e \in \cV$. Recall the stopping time $\tau_1$ from Section \ref{ssec:mart}, where we now take 
$w = x$. Write
\[ q(n,m) 
   = \Pr^{x+e} [ S_{\tau_1} = x + n h_1 + m h_2 ], \quad n, m \in \Z. \]
Then, since $u_\eps$ is discrete harmonic, we can approximate
\eqnsplst{ 
   u_\eps(x+e) - u_\eps(x)
   &= \sum_{|n|, |m| \le C \log (1/\eps)} q(n,m) [ u_\eps(x + n h_1 + m h_2) - u_\eps(x) ] \\
   &\qquad\qquad\qquad + 
     O( \Pr^x [\tau_1 > C \log(1/\eps)] ). }
The error term is $O( \eps^2 )$ if $C$ is chosen large, due to Lemma \ref{lem:tau-decay}. We also have, 
\eqnsplst{
   e
   &= (x+e) - x
   = \sum_{|n|, |m| \le C \log (1/\eps)} q(n,m) [ x + n h_1 + m h_2 - x ] + O( \eps^2 ) \\
   &= \sum_{|n|, |m| \le C \log (1/\eps)} q(n,m) [ n h_1 + m h_2 ] + O( \eps^2 ). }
Hence the statement follows from the statement about $v^{h_i}$, $i=1,2$.
\end{proof}

\section{Appendix}
\label{sec:calculations}

\subsection{Calculations for $180$-degree symmetric lattice}
\label{ssec:xy-lattice}

In this section we give the details of the calculations used in Section \ref{ssec:180}.
We have \cite{Sp-book,KW16} the Fourier representation:
\begin{equation}
\label{e:A(x,y)-formula}
   A(o,(x,y))
   = \frac{1}{2 \pi i} \oint \frac{1}{2 \pi i} \oint \frac{1 - z^x w^y}{D(z,w)} \frac{dw}{w} \frac{dz}{z}, 
\end{equation}
where $D(z,w) = 6 - 2 z - 2 z^{-1} - w - w^{-1}$, and the integrals are over the unit circle.
Writing $t(z) := z + z^{-1}$, we have
\[ w D(z,w)
   = -w^2 + (6 - 2 t(z)) w - 1
   = - (w - w_-(z))(w - w_+(z)), \]
where 
\[ w_{\pm}(z) 
   = \frac{6 - 2 t(z) \pm \sqrt{4 t(z)^2 - 24 t(z) + 32}}{2}
   = 3 - t(z) \pm \sqrt{(2 - t(z))(4 - t(z))}. \]
Noting that $|w_+(z)| > 1$ when $|z| = 1$ and $z \not= 1$, we can perform the $w$-integral 
in \eqref{e:A(x,y)-formula} and get
\begin{equation}
\label{e:A(x,y)-simpler}
\begin{split}
   A(o,(x,y))
   &= \frac{1}{2 \pi i} \oint \frac{1 - z^x w_-(z)^y}{2 \sqrt{2 - t(z)} \sqrt{4 - t(z)}} 
     \frac{dz}{z} \\
   &= \frac{1}{\pi} \int_0^{\pi} \frac{1 - \cos \theta x \left( 3 - 2 \cos \theta 
      - \sqrt{2 - 2 \cos \theta}\sqrt{4 - 2 \cos \theta} \right)^y}{2 \sqrt{2 - 2 \cos \theta} \sqrt{4 - 2 \cos \theta}} \, d\theta. 
\end{split}
\end{equation}
Using that $1 - \cos \theta = 2 \sin^2(\theta/2)$ and 
$4 - 2 \cos \theta = 6 - 4 \cos^2(\theta/2)$, this can be re-written as:
\begin{equation*}
\begin{split}
   &A(o,(x,y)) \\
   &\qquad = \frac{1}{\pi} \int_0^{\pi}  \frac{1 - \cos (2 x \theta/2) 
     \left( 5 - 4 \cos(\theta/2) - 2 \sin(\theta/2) \sqrt{6 - 4 \cos^2(\theta/2)} 
     \right)^y}{4 \sin(\theta/2) \sqrt{6 - 4 \cos^2(\theta/2)}} \, d\theta. 
\end{split}
\end{equation*}
This gives the integrals:
\begin{align*}
    A(o,e)
    &= A(o,(1,0))
    = \frac{1}{\pi} \int_0^{\pi} \frac{\sin(\theta/2)}{2 \sqrt{6 - 4 \cos^2(\theta/2)}} 
      d\theta
    = \frac{\arcsin (\sqrt{6}/3)}{2 \pi}; \\
    A(o,f)
    &= A(o,(0,1))
    = \frac{1}{\pi} \int_0^{\pi} \frac{\sqrt{6 - 4 \cos^2(\theta/2)} - \sin(\theta/2)}{2 \sqrt{6 - 4 \cos^2(\theta/2)}}
    = \frac{\arccos (\sqrt{6}/3)}{\pi}; \\
    A(e,f)
    &= A(o,(1,1))
    = \frac{1}{\pi} \int_0^{\pi} \frac{\cos(\theta) \sqrt{6 - 4 \cos^2(\theta/2)} - 2 \sin(\theta/2) \cos(\theta) + \sin(\theta/2)}{2 \sqrt{6 - 4 \cos^2(\theta/2)}} \\
    &= \frac{\sqrt{2}}{2 \pi}; \\
    A(e,-e)
    &= A(o,(2,0))
    = \frac{1}{\pi} \int_0^{\pi} \frac{2 \sin(\theta/2) \cos^2(\theta/2)}{\sqrt{6 - 4 \cos^2(\theta/2)}}
    = \frac{3 \arcsin(\sqrt{6}/3) - \sqrt{2}}{2 \pi}; \\
    A(f,-f)
    &= A(o,(0,2)) \\
    &= \frac{1}{\pi} \int_0^{\pi} \frac{(3 - 2 \cos(\theta)) \sqrt{6 - 4 \cos^2(\theta/2)} 
      + 4 \sin(\theta/2) (\cos(\theta) - 2)}{\sqrt{6 - 4 \cos^2(\theta/2)}} \\
    &= \frac{6 \arccos(\sqrt{6}/3) - 2 \sqrt{2}}{\pi}. 
\end{align*}

\subsection{Calculations for an isoradial graph}
\label{ssec:isorad-calc}

Let $u, v, w$ denote the missing vertices of the rhombi with half-angles 
$\theta_e, \theta_f, \theta_g$; see Figure \ref{fig:isorad}.
\begin{figure}
    \centering
    \includegraphics[width=0.6\linewidth]{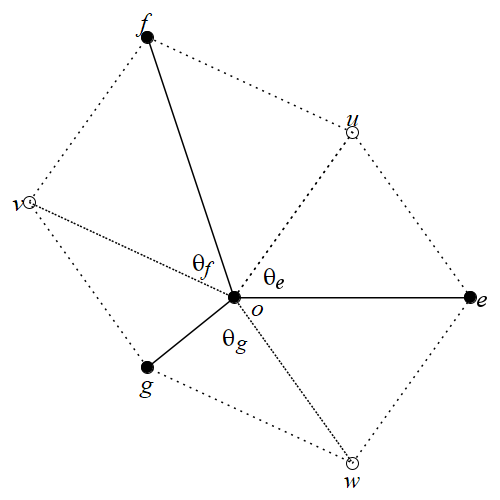}
    \caption{The neighbourhood of $o$ in the isoradial graph.}
    \label{fig:isorad}
\end{figure}

\emph{Computation of $A(o,e), A(o,f), A(o,g)$.} 
In order to apply \cite[Theorem 7.1]{K02}, we define the functions:
\begin{align*} 
   g_o 
   &:= \frac{1}{z} \qquad\quad 
   g_u 
   := g_o \frac{z+ e^{i \theta_e}}{z- e^{i \theta_e}} \qquad\quad
   g_e 
   := g_u \frac{z+ e^{-i \theta_e}}{z- e^{-i \theta_e}} \qquad\quad
   g_v 
   := g_o \frac{z+ e^{i (\theta_e+2 \theta_f)}}{z- e^{i (\theta_e+2\theta_f)}} \\ 
   &\quad
   g_f
   := g_v \frac{z+ e^{i \theta_e}}{z- e^{i \theta_e}} \qquad\quad
   g_w 
   := g_o \frac{z+ e^{-i \theta_e}}{z- e^{-i \theta_e}} \qquad\quad
   g_g 
   := g_w \frac{z+ e^{i (\theta_e+2\theta_f)}}{z- e^{i (\theta_e+2\theta_f)}}
\end{align*}
Consistent angles have to be assigned to each of the poles of $g_e, g_f, g_g$. 
These can be taken to be: \\
$\bullet$ $\theta_e = \pi/3$ and $-\theta_e = -\pi/3$ for the poles of $g_e$; \\
$\bullet$ $\theta_e + 2 \theta_f = 5\pi/6$ and $\theta_e = \pi/3$ for the poles of $g_f$; \\
$\bullet$ $-\theta_e = -\pi/3$ and $-2 \pi + \theta_e + 2 \theta_f = -7\pi/6$ 
for the poles of $g_g$. \\
Then the residue formula of \cite[Theorem 7.1]{K02} gives (note that there is a sign-change
compared to that theorem due to our convention on the sign of the potential kernel):
\begin{align*} 
   A(o,e) 
   &= -\frac{1}{4 \pi i} \left( \theta_e \cdot \mathrm{Res}_{z = e^{i \theta_e}}(g_e(z))
      -\theta_e \cdot \mathrm{Res}_{z = e^{-i \theta_e}}(g_e(z)) \right)
   = \frac{\sqrt{3}}{18} + \frac{\sqrt{3}}{18}
   = \frac{\sqrt{3}}{9} \\
   A(o,f)
   &= -\frac{1}{4 \pi i} \left( (\theta_e + 2 \theta_f) \cdot 
     \mathrm{Res}_{z = e^{i (\theta_e+2\theta_f)}}(g_f(z))
     + \theta_e \cdot \mathrm{Res}_{z = e^{i \theta_e}}(g_f(z)) \right) \\
   &= \frac{5}{12} - \frac{1}{6}
   = \frac{1}{4} \\
   A(o,g)
   &= -\frac{1}{4 \pi i} \left( -\theta_e \cdot 
     \mathrm{Res}_{z = e^{-i \theta_e}}(g_g(z))
     + \left( - 2\pi + \theta_e + 2 \theta_f \right) \cdot 
     \mathrm{Res}_{z = e^{i (\theta_e + 2 \theta_f)}}(g_g(z)) \right) \\
   &= \left( -\frac{1}{3} + \frac{\sqrt{3}}{6} \right)
      + \left( \frac{7}{6} - \frac{7 \sqrt{3}}{12} \right) 
   = \frac{5}{6}-\frac{5 \sqrt{3}}{12}. 
\end{align*}

\emph{Computation of $A(e,f), A(e,g)$.} 
We define the functions:
\begin{align*}
   g'_e 
   &:= \frac{1}{z} \qquad\quad
   g'_u
   := g'_e \frac{z- e^{-i \theta_e}}{z+ e^{-i \theta_e}} \qquad\quad
   g'_f
   := g'_u \frac{z+ e^{i (\theta_e+2\theta_f)}}{z- e^{i (\theta_e+2\theta_f)}} \\
   &\quad 
   g'_w
   := g'_e \frac{z- e^{i \theta_e}}{z+ e^{i \theta_e}} \qquad\quad
   g'_g
   := g'_w \frac{z+ e^{i (\theta_e+2\theta_f)}}{z- e^{i (\theta_e+2\theta_f)}}. 
\end{align*}
The angles for the poles of $g'_f$ and $g'_g$ can be taken to be: \\
$\bullet$ $\pi - \theta_e = 2\pi/3$ and $\theta_e + 2 \theta_f = 5\pi/6$ for 
the poles of $g'_f$; \\
$\bullet$ $-\pi+\theta_e = -2\pi/3$ and $-2 \pi + \theta_e+2\theta_f = -7\pi/6$ 
for the poles of $g'_g$.\\
The residue formulas give:
\begin{align*}
  A(e,f)
  &= -\frac{1}{4 \pi i} \left( \left( \pi - \theta_e \right) \cdot 
     \mathrm{Res}_{z = e^{i (\pi-\theta_e)}}(g'_f(z))
     + (\theta_e + 2 \theta_f) \cdot 
     \mathrm{Res}_{z = e^{i (\theta_e+2\theta_f)}}(g'_f(z)) \right) \\
  &= \left( -\frac{\sqrt{3}}{3}-\frac{2}{3} \right)
     + \left( \frac{5 \sqrt{3}}{12}+\frac{5}{6} \right)
  = \frac{\sqrt{3}}{12}+\frac{1}{6} \\
  A(e,g)
  &= -\frac{1}{4 \pi i} \Big( \left( -\pi + \theta_e \right) \cdot 
     \mathrm{Res}_{z = e^{i (-\pi+\theta_e)}}(g'_g(z)) \\
  &\qquad\qquad + (-2\pi + \theta_e + 2 \theta_f) \cdot 
     \mathrm{Res}_{z = e^{i (\theta_e+2\theta_f)}}(g'_g(z)) \Big) \\
  &= -\frac{1}{3} + \frac{7}{12}
  = \frac{1}{4}. 
\end{align*}

\begin{figure}
    \centering
    \includegraphics[width=0.5\linewidth]{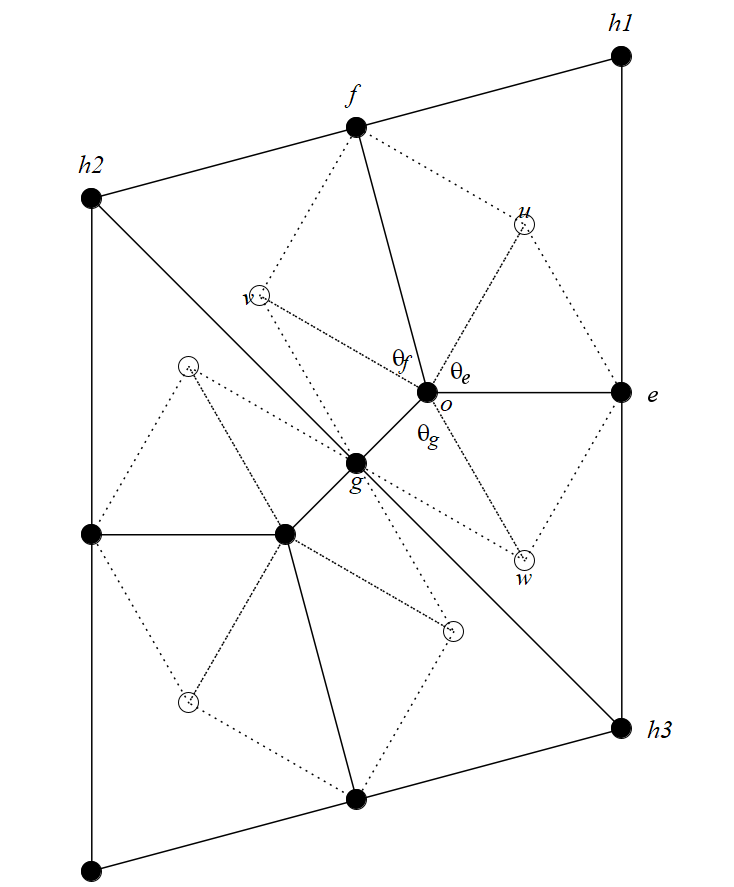}
    \caption{A possible periodic extension: the triangle $h1,h2,h3$ is centrally
    reflected in $g$ to get a possible fundamental domain.}
    \label{fig:isorad-periodic}
\end{figure}

\emph{Computation of $A(f,g)$.} 
We define the functions:
\[ g''_f
   := \frac{1}{z} \qquad\quad
   g''_v
   := g''_f \frac{z- e^{i \theta_e}}{z+ e^{i \theta_e}} \qquad\quad
   g''_g
   := g''_v \frac{z+ e^{-i \theta_e}}{z- e^{-i \theta_e}}. \]
The angles for the poles of $g''_g$ can be taken to be:\\
$\bullet$ $-\pi + \theta_e$ and $-\theta_e$.\\
This gives the value:
\begin{align*}
  A(f,g)
  &= -\frac{1}{4 \pi i} \left( \left( -\pi + \theta_e \right) \cdot 
     \mathrm{Res}_{z = e^{i (-\pi+\theta_e)}}(g''_g(z))
     - \theta_e \cdot 
     \mathrm{Res}_{z = e^{- i \theta_e}}(g''_g(z)) \right) \\
  &= \frac{\sqrt{3}}{3} - \frac{\sqrt{3}}{6}
  = \frac{\sqrt{3}}{6}. 
\end{align*}

\textbf{Funding} 
MWE was supported by a London Mathematical Society Undergraduate Research Bursary 2021.
MWE was also supported by a scholarship from the EPSRC Centre for Doctoral Training in Statistical Applied Mathematics at Bath (SAMBa), under the project EP/S022945/1. 

\bibliographystyle{plain}
\bibliography{refs}

\end{document}